\documentclass{siamart251216}
\usepackage{verbatim}
\usepackage{mathrsfs}
\usepackage{amsmath}
\usepackage{amssymb}
\usepackage{graphicx}
\usepackage{microtype}

\makeatletter
\usepackage{amsfonts}
\usepackage{bm}
\usepackage{color}
\newcommand{\R}{\mathbb{R}}
\newcommand{\E}{\mathbb{E}}

\newcommand{\Piopt}{\Pi_{\mathrm{opt}}}

\newtheorem{thm}{Theorem}
\newtheorem{rem}{Remark}
\newtheorem{cor}{Corollary}
\newtheorem{prop}{Proposition}
\newtheorem{example}{Example}
\newtheorem{lem}{Lemma}
\newtheorem{defn}{Definition}
\newtheorem{assumption}{Assumption}

\usepackage{babel}
\makeatother

\begin{document}
\title{Wasserstein Contraction of Stochastic Systems on Manifolds: A Differential Approach}

\author{Dongjun Wu}

\maketitle


\begin{abstract}
Contraction analysis of dynamical systems turns a global convergence problem into a local one. If the variational system is stable in a uniform sense, then all trajectories of the original system converge toward each other. This paper develops a stochastic analogue for continuous-time stochastic systems, where convergence is measured by Wasserstein distance between probability laws. The key step is to employ mean-square differentiation to construct stochastic variational dynamics. Its uniform $L_2$ stability can then be integrated along curves of initial conditions, yielding Wasserstein contraction of the associated probability laws. This yields criteria formulated in terms of differential Lyapunov functions, and on manifolds, it leads to intrinsic Riemannian formulations.  A distinctive feature of the stochastic setting is that the coupling of the noises becomes a design freedom. To exploit this freedom, we introduce an orthogonal differential coupling that extends the framework beyond synchronous coupling. The resulting criterion explicitly exposes how the curvature of the state space enters the contraction mechanism and, in the orthonormal-frame case, recovers the classical Ricci-curvature structure. Finally, we apply the framework to stochastic control on a compact manifold. By introducing stochastic control input, we overcome a deterministic obstruction to global stabilization. That is, a system that cannot be globally stabilized through smooth time-invariant deterministic feedback is rendered globally contracting in $W_2$, and its steady state distribution can be made arbitrarily concentrated around a prescribed target equilibrium.
\end{abstract}

\section{Introduction}


Contraction analysis provides a differential approach to increm\-ental
stability. Instead of studying convergence to a prescribed equilibrium
or reference trajectory, it asks whether any two solutions of a dynamical
system converge toward one another. For a deterministic system 
\begin{equation} \label{eq:sys-ODE}
\dot{x}=f(t,x),
\end{equation}
such convergence is implied by uniform (asymptotic or exponential) stability of the following \emph{variational dynamics}
\[
\dot{\delta x}=\frac{\partial f}{\partial x}(t,x)\delta x,
\]which characterizes the local or infinitesimal behavior of the system \eqref{eq:sys-ODE}.
Thus the variational system turns a global convergence question to a local one.
This principle dates back
to the work of Demidovich \cite{demidovich1967lectures} and was developed
systematically by Lohmiller and Slotine \cite{lohmiller1998contraction,lohmiller1999}.
It has since led to abundant applications to system analysis
\cite{davydov2022non,kawano2019contraction,ruffer2013convergent,sontag2010contractive},
and synthesis \cite{dani2014observer,manchester2017control,reyes2017tracking,reyes2018virtual,sanfelice2011convergence,wu2024stability}.
A fundamental development is the replacement of a fixed Euclidean
metric by a state-dependent Riemannian, or more generally Finsler
metric \cite{forni2013differential,Wu2021}. 

Most of the existing literature on contraction analysis is concerned with determinist\-ic systems.
In recent years, there has been a growing interest to extend the theories to stochastic systems.
Such an extension is relevant for control systems subject to process noise and
state-dependent stochastic disturbances, for which contraction of individual sample paths is generally not 
expected. Different contraction metrics of stochastic systems 
have thus been considered in the literature, such as mean-square contraction
\cite{pham2009contraction,pham2013stochastic,Koo2026}, moment contraction
\cite{aminzare2022stochastic}, and Wasserstein
contraction of probability measures \cite{bouvrie2019wasserstein}.
These results were typically obtained through independent-noise comparisons and
therefore contain a noise-dependent residual.

More recently, stochastic contraction has been investigated under
common stochas\-tic noise.
For discrete-time systems,
\cite{kawano2023contraction,kawano2025contraction} developed
moment-contraction results.
In continuous time, \cite{kawano2026incremental}
established $W_p$ incremental input-to-state stability for
state-independent, time-varying diffusion and a drift contracting
in a fixed norm.
Almost-sure contraction has been studied in
\cite{MONMARCHE2023,Koo2026}, with \cite{Koo2026} also addressing
$L_p$ stochastic contraction.

In this paper, we study Wasserstein contraction of the probability
laws of continuou\-s-time stochastic systems.
Let $\mu_t$ and $\nu_t$ denote the laws at time $t$ of the same
stochastic system with initial laws $\mu_0$ and $\nu_0$, respectively.
We seek estimates of the form
\begin{equation}
W_p(\mu_t,\nu_t)
\le Ke^{-\lambda t}W_p(\mu_0,\nu_0),
\quad t\ge0,
\label{eq:intro-W2}
\end{equation}
where $W_p$ is the $p$-Wasserstein distance and $K\ge1$,
$\lambda>0$ are constants.

$W_p$ contraction can be characterized in terms of the diffusion semigroup
associated with the underlying stochastic system and is therefore closely linked to properties
of the semigroup. From this perspective,
methods based-on gradient estimates, functional
inequalities, geometric properties of the underlying state space,
and explicit coupling constructions have been extensively developed; see, e.g.,
\cite{baudoin2016wasserstein,cheng2021exponential,kuwada2010duality,
vonrenesse2005transport,wang2020exponential,pascu2016shy}.
These methods have led to powerful contraction results,
especially for diffusions on manifolds, but often rely on the structures of the state space or the systems.

Within the contraction-theoretic literature, Wasserstein contraction
have been established in 
\cite{bouvrie2019wasserstein,kawano2026incremental}.
The estimate in \cite{bouvrie2019wasserstein} contains
a noise-dependent residual arising from an independent-noise
comparison, whereas the residual-free Wasserstein result in
\cite{kawano2026incremental} assumes state-independent diffusion
and a drift contracting in a fixed norm.
In the latter setting, the diffusion terms cancel under
synchronous coupling, so contraction follows from the contraction
property of the drift.
For state-dependent diffusion, this cancellation generally no
longer holds, and derivatives of the diffusion fields enter
the stochastic variational dynamics.

To address these limitations, we develop contraction
conditions using \emph{mean-square variational systems} and an
$L_2$ calculus of stochastic processes.
In Euclidean space, the trajectory of a stochastic system at each fixed time is viewed
as an $L_2$-valued function of its initial condition.
Its mean-square derivative satisfies a stochastic variational
system involving the derivatives of both the drift and the
diffusion fields.
The $L_2$ calculus allows differential contraction estimates
for this system to be integrated along curves of initial
conditions, yielding estimates between coupled solutions.
An optimal coupling of the initial laws then turns these
estimates into Wasserstein contraction.
This construction follows the local-to-global principle of
deterministic contraction analysis. That is, rather than establishing
semigroup estimates directly, we deduce contraction of probability
laws from contraction of the stochastic variational dynamics.
The differential conditions accommodate state-dependent diffusion
and general differential Lyapunov functions, including time- and
state-dependent contraction metrics.
Based on the results in Euclidean space, we extend the construction intrinsically to Riemannian
manifolds and incorporate orthogonal differential coupling,
so that both the geometry of the state space and the choice
of coupling enter the contraction criterion.

Our starting
point is the stochastic variational dynamics. In Euclidean space,
consider an Itô SDE 
\begin{equation}
dX_{t}=f(t,X_{t})\,dt+\sum^{m}_{i=1}\sigma_{i}(t,X_{t})\,dB^{i}_{t}.\label{eq:ItoSDE}
\end{equation}
The mean-square derivative of $X_{t}$ with respect to its initial
condition satisfies the stochastic variational system 
\begin{equation}
\begin{aligned}dY_{t}={} & \frac{\partial f}{\partial x}(t,X_{t})Y_{t}dt+\sum^{m}_{i=1}\frac{\partial\sigma_{i}}{\partial x}(t,X_{t})Y_{t}dB^{i}_{t}.\end{aligned}
\label{eq:Ito_SDE_var}
\end{equation}
The main observation is that the standard path-integration argument
of deterministic contraction can be carried out in the Hilbert space
$L_{2}(\Omega)$: infinitesimal mean-square contraction is integrated
along a curve of initial conditions to obtain a finite mean-square
distance estimate. As a consequence, a differential Lyapunov inequality
of the form
\begin{equation}
\mathcal{A}V\le-2\lambda V\label{eq:cal A-cond}
\end{equation}
yields Wasserstein contraction, where $\mathcal{A}$ is the infinitesimal
generator for the augmented system \eqref{eq:ItoSDE} and \eqref{eq:Ito_SDE_var}
and $V$ the so called Finsler-Lyapunov function. Starting the two
systems from an optimal coupling of their initial laws and evolving
them with the same Brownian motion then turns this estimate into \eqref{eq:intro-W2}.
This directly extends the convential contraction analysis result for
deterministic systems.

This approach shares the same spirit of the so-called 
\emph{derivative-flow} methods \cite{li1992stochastic}, which have long played an important role in stochastic
analysis. In particular, Li \cite{li1992stochastic} studied the solution
flow together with its derivative with respect to the initial condition
and introduced moment-stability conditions based on exponential decay
of the derivative flow. Such conditions were used to study completeness,
semigroup regularity, and geometric and topological properties of the
underlying manifold. The present work adopts a related infinitesimal
viewpoint, but uses contraction of the stochastic variational dynamics
as a differential certificate for contraction of probability laws in
Wasserstein distance.

The same construction admits an intrinsic formulation on a Riemannian
manifold $(M,g)$. For the Stratonovich SDE 
\begin{equation}
dX_{t}=f_{0}(t,X_{t})\,dt+\sum^{m}_{i=1}f_{i}(t,X_{t})\circ dB^{i}_{t},\label{eq:intro-stratonovich}
\end{equation}
the mean square variational dynamics is obtained by complete lifting
(see \cite{Wu2021}) the vector fields $f_{i}$ to the tangent bundle
$TM$. The squared Riemannian length $g(v,v)$ then plays the role
of the Finsler-Lyapunov function. A particularly simple sufficient
condition that emerges from this construction is 
\begin{equation}
\mathcal{L}_{f_{0}}g+\frac{1}{2}\sum^{m}_{i=1}\mathcal{L}^{2}_{f_{i}}g\le-2\lambda g,\label{eq:intro-sync-manifold}
\end{equation}
which implies exponential $W_{2}$ contraction with respect to the
Riemannian distance where $\mathcal{L}_{X}g$ is the metric Lie derivative.
Condition \eqref{eq:intro-sync-manifold} is intrinsic and involves
only Lie derivatives of the metric. It may be viewed as the stochastic
counterpart of the familiar deterministic condition $\mathcal{L}_{f_{0}}g<0$, see e.g.,
\cite{simpson2014contraction}.

Condition \eqref{eq:intro-sync-manifold}, however, corresponds to
a particular choice of coupling, namely synchronous coupling, and
can be conservative. 
The source of this conservatism is
that the law of an Itô diffusion depends on the diffusion covariance
$\sigma\sigma^{\top}$, whereas a particular factorization $\sigma$
is not unique. Different relative rotations of the driving Brownian
motions therefore generate different admissible couplings without
changing the marginal laws. We exploit this freedom infinitesimally.
If $\Omega_{x}(v)\in\mathfrak{so}(m)$ (here $\mathfrak{so}(m)$ represents the set of skew-symmetric matrices of
dimension $m\times m$) denotes the infinitesimal relative
rotation of the noises associated with a tangent displacement $v\in T_{x}M$,
then the differential variation of the diffusion field is modified
from $\nabla_{v}\sigma_{i}$ to 
\begin{equation}
\eta_{i}(v)=\nabla_{v}\sigma_{i}+[\sigma\Omega_{x}(v)]_{i}.\label{eq:intro-diff-coupling}
\end{equation}
We refer to this additional degree of freedom as \emph{orthogonal differential
coupling} (or orthogonal coupling when clear from the context).

Using the orthogonal differential coupling, the resulting manifold criterion
takes the form 
\begin{equation}
\begin{aligned}g(\nabla_{v}\bar{f},v)+\frac{1}{2}\sum^{m}_{i=1}\Big(-R(\sigma_{i},v,v,\sigma_{i})+|\eta_{i}|^{2}_{g}\Big)\le-\lambda|v|^{2}_{g},\end{aligned}
\label{eq:intro-general-manifold}
\end{equation}
where $\bar{f}=f+\frac{1}{2}\sum^{m}_{i=1}\nabla_{\sigma_{i}}\sigma_{i}$, and $R$ is the curvature tensor.
The coupling variable $\Omega$ can compensate
for changes of the diffusion frame which represent only a rotation
of the driving noise.

We see that curvature plays a crucial role for Wasserstein contraction on manifolds. This is most transparent
in the following special case.
When $\{\sigma_{i}\}$ is an orthonormal
frame, one can choose $\Omega$ so that $\eta_{i}=0$. Condition \eqref{eq:intro-general-manifold}
then reduces to 
\begin{equation}
g(\nabla_{v}\bar{f},v)-\frac{1}{2}{\rm Ric}(v,v)\le-\lambda|v|^{2}_{g},\label{eq:intro-ricci}
\end{equation}where ${\rm Ric}$ is the Ricci curvature.
The left hand side is the familiar Bakry-Émery/Ricci curvature
structure underlying Wasserstein contraction of diffusion semigroups
\cite{bakry1985diffusions,kuwada2010duality,vonrenesse2005transport}.
Thus the classical curvature criterion appears here as a special case
of a differential contraction condition equipped with an appropriate
infinitesimal coupling. But it is derived from a rather different
perspective. %

The \textbf{contributions} of the paper can be summarized as follows.
\begin{enumerate}
\item We develop a differential approach for Wasserstein contraction
and establish condition \eqref{eq:cal A-cond}. The main technique
is to introduce a mean-square variation of the stochastic system,
mimicking the conventional deterministic variational systems. The
Euclidean result allows general Finsler-Lyapunov function and, in
particular, time- and state-dependent contraction metrics.
\item We formulate the construction intrinsically on Riemannian manifolds
using stochastic prolongation to the tangent bundle. This gives the
Lie-derivative criterion \eqref{eq:intro-sync-manifold}, providing
a direct stochastic analogue of deterministic Rieman\-nian contraction.
\item We introduce orthogonal differential coupling as an additional degree
of freedom and derive the curvature-dependent condition \eqref{eq:intro-general-manifold}.
The construction explains the conservatism of synchronous coupling,
recovers the Ricci-curvature criterion in the orthonormal-frame case,
and yields sharp contraction certifica\-tes for representative stochastic systems
on manifolds such as the sphere, $SO(n)$ and the space of positive definite matrices.
\item We illustrate the control implications through noise-assisted stabilization
on a compact manifold, where a smooth deterministic feedback is prevented
from global stabilizatizing an equilibrium. By adding a suitably shaped
state-dependent diffusion noise, the resulting closed-loop system
becomes globally exponentially contracting in $W_{2}$ and possesses
a unique invariant distribution that can be made arbitrarily concentrated
around the desired target while retaining a strictly positive contraction
rate. 
\end{enumerate}
The presentation follows a bottom-up route. We first develop the Euclidean
theory, where the mean-square variational arguments are most transparent.
We then employ the construction to Riemannian manifolds and obtain the
intrinsic synchronous criterion \eqref{eq:intro-sync-manifold}. The
limitations of synchronous coupling motivate the subsequent differential-coupling
theory, first through its infinitesimal interpretation and then in
its intrinsic manifold form. Finally, the geometric results are applied
to noise-assisted stabilization on the 2-sphere.

\section{Problem Formulation}

\label{sec:problem}

We first formulate the contraction problem in Euclidean space, where
we allow general differential Lyapunov functions. The intrinsic Riemannian
formulation will be introduced later as a geometric specialization.

Consider the time-varying Itô stochastic differential equation \eqref{eq:ItoSDE}
with $X_{t}\in\mathbb{R}^{n}$, where $B=(B^{1},\ldots,B^{m})$ is
a standard $m$-dimensional Brownian motion. For an initial condition
$X_{0}=x$, we denote the solution at time $t\ge0$ by $X^{x}_{t}$
and its law by $\mu^{x}_{t}$. By slightly abuse of notation, for
an initial probability measure $\mu_{0}$, denote the law of $X_{t}$
as $\mu_{t}$. 

For $p\ge1$, let $\mathcal{P}_{p}(\mathbb{R}^{n})$ denote the set
of Borel probability measures with finite $p$th moment. The $p$-Wasserstein
distance is 
\begin{equation}
W^{p}_{p}(\mu,\nu):=\inf_{\pi\in\Pi(\mu,\nu)}\int_{\mathbb{R}^{n}\times\mathbb{R}^{n}}|x-y|^{p}d\pi(x,y),\label{eq:Wp-def}
\end{equation}
where $\Pi(\mu,\nu)$ denotes the set of couplings of $\mu$ and $\nu$,
see for example \cite{villani2021topics}.

Our basic objective is to establish exponential contraction inequality
\eqref{eq:intro-W2} for arbitrary $\mu_{0},\nu_{0}\in\mathcal{P}_{p}(\mathbb{R}^{n})$,
with constants $K\ge1$ and $\lambda>0$ independent of the initial
measures. When $K=1$, the system is said to be $W_{p}$ contractive
with rate $\lambda$.

The approach taken in this paper is differential, inspired by contraction
analysis developed mainly for deterministic systems \cite{bulloBook2026}.
Along a solution $X_{t}$ of \eqref{eq:ItoSDE}, an infinitesimal
displacement $Y_{t}$ satisfies the stochastic variational equation
\eqref{eq:Ito_SDE_var}. We therefore consider a nonnegative function,
\begin{equation}
V:[0,\infty)\times\mathbb{R}^{n}\times\mathbb{R}^{n}\longrightarrow[0,\infty).\label{eq:problem-V}
\end{equation}
Typical assumptions will require, for some $p\ge1$, 
\begin{equation}
c_{1}|y|^{p}\le V(t,x,y)\le c_{2}|y|^{p}.\label{eq:problem-V-bounds}
\end{equation}
Such a function plays the role of a stochastic Finsler-Lyapunov
function \cite{forni2013differential}. 
In particular, this formulation
includes 
$p=1$, leading to $W_{1}$ theory, as well as the quadratic
case $p=2$, relevant to $W_{2}$.

The first question addressed in the paper is therefore the following.

\medskip{}
\noindent \textbf{Problem 1.} \emph{Find differential conditions
on $V(t,x,y)$ and on the coefficients of \eqref{eq:ItoSDE} which
imply contraction of the induced probability laws in $W_{p}$, and
in particular in $W_{1}$ and $W_{2}$.} \medskip{}

A particularly important specialization is the quadratic form 
\begin{equation}
V(t,x,y)=y^{\top}P(t,x)y,\quad P(t,x)=P(t,x)^{\top}\succ0.\label{eq:problem-quadratic-V}
\end{equation}
This is the stochastic analogue of a Riemannian contraction metric.
It is this quadratic structure that admits a natural intrinsic extension
to nonlinear state spaces.
Indeed, in the second part of the paper we consider a connected Riemannian
manifold $(M,g)$ and a Stratonovich diffusion \eqref{eq:intro-stratonovich}
with $X_{t}\in M$. Let $d_{g}$ denote the Riemannian distance and
define 
\begin{equation}
W^{2}_{2}(\mu,\nu):=\inf_{\pi\in\Pi(\mu,\nu)}\int_{M\times M}d^{2}_{g}(x,y)d\pi(x,y).\label{eq:W2g-def}
\end{equation}
The differential Lyapunov function is then fixed by the geometry:
\begin{equation}
V(x,v)=g_{x}(v,v),\quad v\in T_{x}M.\label{eq:problem-riemannian-V}
\end{equation}

\medskip{}
\noindent \textbf{Problem 2.} \emph{For the stochastic system \eqref{eq:intro-stratonovich} on Riemannian
	manifolds, characterize differential conditions
on the drift, diffusion vector fields, metric, and possibly the coupling
of the driving noises, which guarantee $W_{2}$ contraction.} 

\medskip{}
Since Problem 1 provides the foundation for Problem 2, we first address Problem 1 in the next section.

\section{Wasserstein Contraction in Euclidean Spaces }\label{sec:main}

We first develop the differential contraction theory in Euclidean
space. The construction follows the same local-to-global principle
as deterministic contraction analysis: an infinitesimal displacement
is propagated by a variational system, and its decay is integrated
along a curve of initial conditions to obtain a finite-distance estimate.
The essential difference is that regular differentiability is lost
in the stochastic case. We therefore need a weaker notion of derivative
called \emph{mean-square derivative}. This notion is not new and well-known
in stochastic systems, see for example, \cite[Chapter 7]{arnold1974stochastic},
or \cite[Chapter V]{Pro1992}. 

Throughout this section, we consider the Itô diffusion \eqref{eq:ItoSDE}
introduced in Section~\ref{sec:problem}.

\subsection{Mean-square variational dynamics}

There seems to be no consensus on what minimal regularity assumptions
should be imposed to guarantee existence of mean square derivative
with respect to initial conditions. The following set of conditions
will suffice our needs.

\begin{assumption} \label{ass:regularity} For every $T>0$: 
\begin{enumerate}
\item $f,\sigma_{i},\frac{\partial f}{\partial x},\frac{\partial\sigma_{i}}{\partial x}$
are jointly measurable in $(t,x)$, and $f(t,\cdot),\sigma_{i}(t,\cdot)$
are $C^{1}$ for almost every $t$; 
\item 
\begin{equation}
\int^{T}_{0}\left(|f(t,0)|^{2}+\sum^{m}_{i=1}|\sigma_{i}(t,0)|^{2}\right)dt<\infty;
\end{equation}
\item there exists $L_{T}>0$ such that 
\begin{equation}
\sup_{\substack{t\in[0,T]\\
x\in\R^{n}
}
}\left(\left\Vert \frac{\partial f(t,x)}{\partial x}\right\Vert +\sum^{m}_{i=1}\left\Vert \frac{\partial\sigma_{i}(t,x)}{\partial x}\right\Vert \right)\le L_{T};\label{eq:derivative-bound}
\end{equation}
\item $\frac{\partial f}{\partial x}$ and $\frac{\partial\sigma_{i}}{\partial x}$
are uniformly continuous in $x$, uniformly for $t\in[0,T]$. 
\end{enumerate}
\end{assumption}

In particular, the coefficients are globally Lipschitz on every finite
time interval, and hence \eqref{eq:ItoSDE} admits a unique nonexplosive
strong solution.
\begin{defn}
The solution $X^{x}_{t}$ of the SDE \eqref{eq:ItoSDE} is \emph{mean-square
differentiable} with respect to $x$ in the direction $y\in\R^{n}$
if there exists $Y^{x,y}_{t}\in L^{2}(\Omega;\R^{n})$ such that 
\begin{equation}
\lim_{h\to0}\E\left|\frac{X^{x+hy}_{t}-X^{x}_{t}}{h}-Y^{x,y}_{t}\right|^{2}=0.\label{eq:mean square der}
\end{equation}
\end{defn}

The derivative is governed by the stochastic analogue of the usual
variational equation.
\begin{prop}
\label{prop:msdiff} Under Assumption~\ref{ass:regularity}, for
every $T>0$ and $x,y\in\R^{n}$ there exists an adapted process $Y^{x,y}$
such that 
\begin{equation}
\E\sup_{0\le t\le T}|Y^{x,y}_{t}|^{2}<\infty
\end{equation}
and 
\begin{equation}
\lim_{h\to0}\E\sup_{0\le t\le T}\left|\frac{X^{x+hy}_{t}-X^{x}_{t}}{h}-Y^{x,y}_{t}\right|^{2}=0.\label{eq:msdiff}
\end{equation}
Moreover, $Y^{x,y}$ is the unique solution of \eqref{eq:Ito_SDE_var}
with initial condition $Y^{x,y}_{0}=y$. The map $(x,y)\mapsto Y^{x,y}$
is continuous from $\R^{n}\times\R^{n}$ into $L_{2}(\Omega;C([0,T];\R^{n}))$. 
\end{prop}

The proof is given in Appendix~\ref{app:msdiff}. It uses difference
quotients together with standard $L_{2}$ stability estimates for
stochastic differential equations. Proposition~\ref{prop:msdiff}
provides the stochastic counterpart of the classical variational system
used in determini\-stic contraction analysis.

For convenience, we write \eqref{eq:ItoSDE} and \eqref{eq:Ito_SDE_var}
together
\begin{equation}
\begin{aligned}dX_{t} & =f(t,X_{t})dt+\sum^{m}_{i=1}\sigma_{i}(t,X_{t})dB^{i}_{t},\\
dY_{t} & =\frac{\partial f}{\partial x}(t,X_{t})Y_{t}dt+\sum^{m}_{i=1}\frac{\partial\sigma_{i}}{\partial x}(t,X_{i})Y_{t}dB^{i}_{t},
\end{aligned}
\label{eq:lifted}
\end{equation}
and call this the \emph{mean-square variational system}.

The mean-square variational system describes only infinitesimal perturbations.
Nevertheless, path integration estimation as in deterministic contraction
theory holds in the $L_{2}$ sense, which in turn leads to incremental
estimations. %

\begin{lem}
\label{lem:FTC} Let $\gamma\in C^{1}([0,1];\R^{n})$. Under Assumption~\ref{ass:regularity},
\begin{equation}
X^{\gamma(1)}_{t}-X^{\gamma(0)}_{t}=\int^{1}_{0}Y^{\gamma(s),\gamma'(s)}_{t}\,ds\label{eq:L2FTC}
\end{equation}
holds in $L^{2}(\Omega;\R^{n})$. Consequently, 
\begin{equation}
\E|X^{\gamma(1)}_{t}-X^{\gamma(0)}_{t}|^{2}\le\int^{1}_{0}\E|Y^{\gamma(s),\gamma'(s)}_{t}|^{2}\,ds.\label{eq:L2FTC-Jensen}
\end{equation}
\end{lem}
\begin{proof}
Proposition~\ref{prop:msdiff} implies that $x\mapsto X^{x}_{t}$
is continuously Fréchet differentiable as a map from $\R^{n}$ into
$L^{2}(\Omega;\R^{n})$, with directional derivative $Y^{x,y}_{t}$.
Hence $s\mapsto X^{\gamma(s)}_{t}$ is continuously differentiable
and 
\begin{equation}
\frac{d}{ds}X^{\gamma(s)}_{t}=Y^{\gamma(s),\gamma'(s)}_{t}\quad\text{in }L^{2}.
\end{equation}
The Banach-space fundamental theorem of calculus (see e.g., \cite[Chapter XIII]{lang1993real})
gives \eqref{eq:L2FTC}; Jensen's inequality gives \eqref{eq:L2FTC-Jensen}.
\end{proof}

Let $\mathcal{A}$ denote the infinitesimal generator of the variational
system \eqref{eq:lifted}. For 
\[
V\in C^{1,2,2}\bigl([0,\infty)\times\R^{n}\times(\R^{n}\setminus\{0\})\bigr),
\]
we have 
\begin{equation}
\begin{aligned}\mathcal{A}V={} & \partial_{t}V+(\nabla_{x}V)^{\top}f+(\nabla_{y}V)^{\top}\frac{\partial f}{\partial x}y\\
 & +\frac{1}{2}\sum^{m}_{i=1}\begin{bmatrix}\sigma_{i}\\
\frac{\partial\sigma_{i}}{\partial x}y
\end{bmatrix}^{\!\top}D^{2}_{(x,y)}V\begin{bmatrix}\sigma_{i}\\
\frac{\partial\sigma_{i}}{\partial x}y
\end{bmatrix},
\end{aligned}
\label{eq:lifted-generator}
\end{equation}
where all coefficients are evaluated at $(t,x)$.

\subsection{Wasserstein Contraction in Euclidean Spaces}

Let $\mu_{t}$ and $\nu_{t}$ denote the laws at time $t$ for the
system \eqref{eq:ItoSDE} generated from two initial probability measures
$\mu_{0}$ and $\nu_{0}$. When $\mu_{0}$ is a Dirac measure, say
$\mu_{0}=\delta_{x_{0}}$, to highlight the dependence on the initial
condition, we shall denote $\mu_{t}$ as $\mu^{x_{0}}_{t}$.

\begin{defn}[$W_{p}$ contraction]
The system is said to be Wasserstein contracting (or $W_{p}$ contracting
in short) if there exist two positive constants $K$, $\lambda$ such
that, for any $X_{0}\sim\mu_{0}$, $Y_{0}\sim\nu_{0}$, and $X_{t}\sim\mu_{t}$,
$Y_{t}\sim\nu_{t}$, there holds
\begin{equation}
W_{p}(\mu_{t},\nu_{t})\le Ke^{-\lambda t}W_{p}(\mu_{0},\nu_{0}),\quad t\ge0.\label{eq:w2-goal}
\end{equation}
\end{defn}

The following result provides a sufficient condition for Wasserstein contraction in Euclidean spaces.  
\begin{thm}
\label{thm:main} Suppose Assumption~\ref{ass:regularity} holds
and let $1\le p\le2$. Assume that there exist a function $V\in C([0,\infty)\times\R^{n}\times\R^{n})\cap C^{1,2,2}([0,\infty)\times\R^{n}\times(\R^{n}\backslash\{0\}))$
and constants $c_{1},c_{2},\lambda>0$ such that 
\begin{equation}
c_{1}|y|^{p}\le V(t,x,y)\le c_{2}|y|^{p}\label{eq:V-bounds}
\end{equation}
and 
\begin{equation}
\mathcal{A}V(t,x,y)\le-\lambda V(t,x,y)\label{eq:V-decay}
\end{equation}
for all $t\ge0$ and $x\in\R^{n}$, $y\in\mathbb{R}^{n}\backslash\{0\}$.
Then the system \eqref{eq:ItoSDE} is $W_{p}$ contracting.
\end{thm}

Conditions \eqref{eq:V-bounds} and \eqref{eq:V-decay} have
already appeared in \cite{Koo2026} in the context of $L_p$ contraction, where they were derived
using stochastic flow tools \cite{kunita1990stochastic}.
Here, we give a more self-contained derivation based on mean-square differentiation.
Since this approach will be used repeatedly throughout the manuscript, we include the complete proof below.

\begin{proof}
	The proof relies essentially on Dynkin's formula and Lemma \ref{lem:FTC}.

\emph{Step 1.} We treat the $p=2$ case first. If $V$ is $C^{2}$
on the whole of $[0,\infty)\times\mathbb{R}^{n}\times\mathbb{R}^{n}$,
then Dynkin's formula (see for example \cite[Chapter II]{Kushner1967})
together with the inquality \eqref{eq:V-decay} imply
\begin{equation}
\mathbb{E}V(t,X_{t},Y_{t})\le e^{-\lambda t}\mathbb{E}V(0,X_{0},Y_{0}).\label{eq:ineq-V}
\end{equation}
However, in our case, $V(\cdot,\cdot,y)$ is only continuous at $y=0$.
This can be easily handled through the following locolization procedure. 

Let $Y_{0}\neq0$. The variational equation (second line of \eqref{eq:lifted})
with initial condition $Y_{0}$ has the representation $Y_{t}=\Phi_{t}Y_{0}$,
where $\Phi_{t}$ is the its fundamental matrix which is almost surely
invertible. Consequently, $Y_{t}\neq0$ for all $t\ge0$ almost surely.

For $k\ge1$, define the stopping time 
\[
\tau_{k}:=\inf\left\{ t\ge0:|Y_{t}|\le k^{-1}\ \text{or}\ |X_{t}|+|Y_{t}|\ge k\right\} .
\]
On $[0,\tau_{k}]$, the process remains in a compact subset of $\R^{n}\times(\R^{n}\setminus\{0\})$,
where $V$ is $C^{1,2,2}$. Dynkin's formula applied to $e^{\lambda(t\wedge\tau_{k})}V(t\wedge\tau_{k},X_{t\wedge\tau_{k}},Y_{t\wedge\tau_{k}})$
therefore gives 
\[
\E\!\left[e^{\lambda(t\wedge\tau_{k})}V(t\wedge\tau_{k},X_{t\wedge\tau_{k}},Y_{t\wedge\tau_{k}})\right]\le V(0,x,y).
\]
Since $\tau_{k}\uparrow\infty$ almost surely as $k\to\infty$, Fatou's
lemma yields \eqref{eq:ineq-V}. 

For $Y_{0}=0$, the same estimate is trivial since $Y^{x,0}_{t}\equiv0$.
Using \eqref{eq:V-bounds}, we then arrive at
\begin{equation}
\E|Y^{x,y}_{t}|^{2}\le\frac{c_{2}}{c_{1}}e^{-\lambda t}|y|^{2}\label{eq:tangent-contraction}
\end{equation}
for every $t\ge0$, $x\in\mathbb{R}^{n}$, $y\in\mathbb{R}^{n}$.

Now for any two deterministic initial points $x_{1},x_{2}$, take
the segment $\gamma(s)=(1-s)x_{1}+sx_{2}$. Lemma~\ref{lem:FTC}
gives 
\begin{equation}
\E|X^{x_{2}}_{t}-X^{x_{1}}_{t}|^{2}\le\frac{c_{2}}{c_{1}}e^{-2\lambda t}|x_{2}-x_{1}|^{2}.\label{eq:point-contraction}
\end{equation}

\emph{Step 2}. Choose an optimal initial coupling $\pi_{0}\in\Piopt(\mu_{0},\nu_{0})$,
sample $(X_{0},\widetilde{X}_{0})\sim\pi_{0}$ independently of $B$,
and evolve the two states with the same Brownian motion. Conditioning
on $(X_{0},\widetilde{X}_{0})$ and using \eqref{eq:point-contraction},
\begin{align}
\E|X_{t}-\widetilde{X}_{t}|^{2} & \le\frac{c_{2}}{c_{1}}e^{-2\lambda t}\E|X_{0}-\widetilde{X}_{0}|^{2}\nonumber \\
 & =\frac{c_{2}}{c_{1}}e^{-2\lambda t}W^{2}_{2}(\mu_{0},\nu_{0}).
\end{align}
The law of $(X_{t},\widetilde{X}_{t})$ is an admissible coupling
of $\mu_{t}$ and $\nu_{t}$, so $W^{2}_{2}(\mu_{t},\nu_{t})\le\E|X_{t}-\widetilde{X}_{t}|^{2}$.
Taking square roots proves \eqref{eq:w2-goal}. 

\emph{Step 3}. It remains consider $1\le p<2$. Since the underlying
measure is a probability measure, $L_{2}\subset L_{p}$ and $\|f\|_{p}\le\|f\|_{2}$.
Hence mean-square differentiability implies $L_{p}$-differentiability
for every $1\le p<2$, and the proof for $p=2$ carries over verbatim.

\end{proof}

\begin{rem}
Note carefully that the function $V(\cdot,\cdot,y)$ is required to
be $C^{2}$ for the third variable only on $\mathbb{R}^{n}\backslash\{0\}$. 
\end{rem}

For linear systems, there is a simple characterization for $W_2$ contraction.
Consider
\begin{equation}
dX_{t}=AX_{t}dt+\sum^{m}_{i=1}C_{i}X_{t}dB^{i}_{t},\label{eq:GBM}
\end{equation}
where $A,C_{i}\in\mathbb{R}^{n\times n}$ are some constant matrices.
We have the following result.
\begin{cor}
The system \eqref{eq:GBM} is $W_{2}$ contracting if and only if
there exist $P\succ0$ and $Q\succ0$, such that 
\begin{equation}
A^{\top}P+PA+\sum^{m}_{i=1}C^{\top}_{i}PC_{i}\preceq-Q.\label{eq:L-ineq}
\end{equation}
\end{cor}
\begin{proof}
Sufficiency follows from Theorem \ref{thm:main} by taking
$V(t,x,y)=y^{\top}Py$. 
We prove necessity by contradiction. It is well-known that the system
\eqref{eq:GBM} is mean-square exponentially stable if and only if
the inequality holds for some positive definite matrices $P$ and
$Q$; see for example \cite[Corollary 11.4.14]{arnold1974stochastic}.
Suppose, to the contrary, that no such $P$ and $Q$ exist. Then the
system is not mean-square exponentially stable. Let $X_{0}=x_{0}$,
$X_{t}\sim\mu_{t}$ and choose $\nu_{0}=\delta_{0}$. Since the origin
is an equilibrium, $\nu_{t}=\delta_{0}$. Now $W_{2}$ contraction
implies
\[
\mathbb{E}|X_{t}|^{2}=W_{2}(\mu_{t},\delta_{0})^{2}\le Ce^{-2\lambda t}|x_{0}|^{2}.
\]
Since this holds for all $x_{0}\in\mathbb{R}^{n}$, this leads to
a contradiction.
\end{proof}

\subsection{Stochastic Matrix Measure and $W_{1}$ contraction}

Matrix measure (or matrix logarithmic norm) is widely used in analyzing
deterministic contraction \cite{Russo2010,sontag2010contractive}. 
Stochastic extensions of the matrix measure have been considered in
\cite{ahmad2008stochastic,aminzare2022stochastic}. 
Here, we define stochastic matrix measure as follows.

Let $A$ and $C,\cdots,C_{m}$ be some matrices in $\mathbb{R}^{n\times n}$.
Let $|\cdot|_{*}$ be some norm in $\mathbb{R}^{n}$ which is smooth
on $\mathbb{R}^{n}\backslash\{0\}$. Define 
\begin{equation} \label{eq:matrix-measure}
\begin{aligned}
\mu_{|\cdot|_{*}}(A; & C_{1},\cdots,C_{m})\\
= & \sup_{|y|_{*}=1}\left\{ y^{\top}A^{\top}\frac{\partial|y|_{*}}{\partial y}+\frac{1}{2}\sum^{m}_{i=1}y^{\top}C^{\top}_{i}\frac{\partial^{2}|y|_{*}}{\partial y^{2}}C_{i}y\right\} .
\end{aligned}
\end{equation}
The definition reduces to the conventional one by setting $C_{1}=\cdots=C_{m}=0$. 
Below we show that this stochastic matrix measure arises naturally by taking Finsler-Lyapunov function as
$V(t,x,y) = |y|_*$.

\begin{prop}
Let $|\cdot|_{*}$ be an Euclidean norm which is $C^{2}$ on $\mathbb{R}^{n}\backslash\{0\}$.
If there exists a constant $C>0$, such that
\[
\mu_{|\cdot|_{*}}\left(\frac{\partial f}{\partial x};\frac{\partial\sigma_{1}}{\partial x},\cdots,\frac{\partial\sigma_{m}}{\partial x}\right)(x)<-C,
\]
for all $x\in\mathbb{R}^{n}$, then the system \eqref{eq:ItoSDE}
is $W_{1,|\cdot|_{*}}$ contracting. 

\end{prop}
\begin{proof}
Consider $V(t,x,y)=|y|_{*}$. Then it is easy to check that, outside
$y=0$, there holds
\begin{align*}
 & \mathcal{A}V(t,x,y)\\
= & \,y^{\top}\left(\frac{\partial f}{\partial x}\right)^{\top}\frac{\partial|y|_{*}}{\partial y}+\frac{1}{2}\sum^{m}_{i=1}y^{\top}\left(\frac{\partial\sigma_{i}}{\partial x}\right)^{\top}\frac{\partial^{2}|y|_{*}}{\partial y^{2}}\frac{\partial\sigma_{i}}{\partial x}y\\
\le & \,\mu_{|\cdot|_{*}}\left(\frac{\partial f(t,x)}{\partial x};\frac{\partial\sigma_{1}(t,x)}{\partial x},\cdots,\frac{\partial\sigma_{m}(t,x)}{\partial x}\right)|y|_{*}\\
\le & \,-C|y|_{*}.
\end{align*}
This yields $W_{1,|\cdot|_{*}}$ contraction. 
\end{proof}

\begin{rem}
	The stochastic logarithmic norm and stochastic logarithmic Lipschitz constant 
	introduced in \cite[Definition 2, linear case]{ahmad2008stochastic} and
	\cite[Definition 4, $l=1$]{aminzare2022stochastic} are equivalent to the stochastic matrix measure defined
	above, as follows directly from It\^o's formula. The present definition is obtained from a
	different perspective. It arises naturally by applying the
	Finsler-Lyapunov differential framework to the stochastic variational
	dynamics. Moreover, the formula \eqref{eq:matrix-measure} is also more amenable to
	computation, since it gives a pointwise expression in terms of the
	drift and diffusion matrices and the derivatives of the underlying
	norm, without resorting to short-time stochastic expansions.
	$L_1$ contraction results have been obtained in \cite{Koo2026} in terms of similar
	conditions,
	though the authors did not define explicitly stochastic matrix measure. 
\end{rem}

\begin{rem}
If $|\cdot|_{*}$ is the usual Euclidean $2$-norm $|\cdot|$, and
that the system \eqref{eq:ItoSDE} is scalar, then it can be readily
checked that
\[
\mu_{|\cdot|}\left(\frac{\partial f}{\partial x},\frac{\partial\sigma}{\partial x}\right)=\mu_{|\cdot|}\left(\frac{\partial f}{\partial x},0\right).
\]
That is, a scalar SDE is $W_{1,|\cdot|}$ contracting if the deterministic
system $\dot{x}=f(x)$ is contracting in the Euclidean norm. This
is a rather interesting phenomenon which does not hold for $W_{2}$ contraction. 
\end{rem}

\section{Wasserstein Contraction on Riemannian Manifolds}

We now extend the differential construction of the previous section
to manifolds. Let $(M,g)$ be a connected Riemannian manifold and
consider the Stratonovich SDE 
\begin{equation}
dX_{t}=f_{0}(t,X_{t})\,dt+\sum^{m}_{i=1}f_{i}(t,X_{t})\circ dB^{i}_{t},\label{eq:manifold-stratonovich}
\end{equation}
where $f_{0},f_{1},\ldots,f_{m}$ are vector fields on $M$. The Stratonovich
representation is preferred on manifolds \cite{ikeda1989stochastic},
as it is intrinsic and obeys the ordinary chain rule.

Throughout this section, $d_{g}$ and $|\cdot|_{g}$ denote respectively
the Riemannian distance and norm induced by $g$. The corresponding
$2$-Wasserstein distance was defined in \eqref{eq:W2g-def}.

We impose the following standing regularity assumption.

\begin{assumption} \label{ass:manifold} The Riemannian manifold
$(M,g)$ is geodesically complete. For every finite $T>0$, the vector
fields $f_{i}(t,\cdot)$ are $C^{2}$ in the state variable and continuous
in time. Their coordinate representations satisfy, on every relatively
compact chart, the local regularity conditions required in Assumption~\ref{ass:regularity}.
\end{assumption}

\subsection{Stochastic prolongation on the tangent bundle}

As in deterministic contraction theory, the first step is to describe
the evolution of an infinitesimal displacement. On a manifold this
is naturally achieved by lifting the dynamics from $M$ to the tangent
bundle $TM$.

Let $W$ be a vector field on $M$. Its \emph{complete lift}, denoted
by $W^{\dagger}$, is the vector field on $TM$ whose flow is the
tangent map of the flow of $W$. More precisely, if $\phi_{t}$ is
the flow of $W$, then the flow of $W^{\dagger}$ is $D\phi_{t}$.
In local coordinates $(x^{a},y^{a})$ on $TM$, the complete lift
reads 
\begin{equation}
W^{\dagger}=W^{a}(x)\frac{\partial}{\partial x^{a}}+\frac{\partial W^{a}}{\partial x^{b}}(x)y^{b}\frac{\partial}{\partial y^{a}}.\label{eq:complete-lift-coordinate}
\end{equation}
We refer to \cite{crampin1986applicable,Wu2021} for further details.

Applying the complete lift to each vector field in \eqref{eq:manifold-stratonovich}
gives the \emph{stochastic prolongation} 
\begin{equation}
dZ_{t}=f^{\dagger}_{0}(t,Z_{t})\,dt+\sum^{m}_{i=1}f^{\dagger}_{i}(t,Z_{t})\circ dB^{i}_{t},\label{eq:manifold-prolongation}
\end{equation}
where $Z_{t}\in TM$. In a local coordinate chart, \eqref{eq:manifold-prolongation}
takes the familiar form 
\begin{align}
dX_{t} & =f_{0}(t,X_{t})\,dt+\sum^{m}_{i=1}f_{i}(t,X_{t})\circ dB^{i}_{t},\label{eq:manifold-prolongation-x}\\
dY_{t} & =\frac{\partial f_{0}}{\partial x}(t,X_{t})Y_{t}\,dt+\sum^{m}_{i=1}\frac{\partial f_{i}}{\partial x}(t,X_{t})Y_{t}\circ dB^{i}_{t}.\label{eq:manifold-prolongation-y}
\end{align}
Thus the second component is precisely the local stochastic variational
dynamics.

The following identity connects the complete lift with the Riemannian
metric.
\begin{lem}
\label{lem:lift}Let $W$ be a vector field on $(M,g)$ and $W^{\dagger}$
the complete lift of $W$. Then for $h(v)=g(v,v)$, there holds 
\[
W^{\dagger}(h)(v)=(\mathcal{L}_{W}g)(v,v)
\]where $\mathcal{L}_W g$ is the metric Lie derivative of $g$ along $W$.
\end{lem}

\begin{proof}
Let $\phi_{t}$ be the flow of $W$. Then the flow of $W^{\dagger}$
is $D\phi_{t}$. Hence
\begin{align*}
W^{\dagger}(h)(v) & =\left.\frac{d}{dt}\right|_{t=0}h(D\phi_{t}(v))\\
 & =\left.\frac{d}{dt}\right|_{t=0}g(D\phi_{t},D\phi_{t})\\
 & =(\mathcal{L}_{W}g)(v,v),
\end{align*}in which the last equality follows from the definition of metric Lie derivative, see
\eqref{eq:metric-Lie}.
\end{proof}

The stochastic prolongation has the same interpretation as the mean-square
variational dynamics of Section~\ref{sec:main}.
\begin{lem}
\label{lem:mani-esti1} Let $(x,v)\in TM$ and $\gamma:(-\epsilon,\epsilon)\to M$
a $C^{1}$ curve with $\gamma(0)=x$ and $\gamma'(0)=v$. Denote denote
the second component of the prolonged system \eqref{eq:manifold-prolongation}
as $Y^{x,v}_{t}$ with initial condition $(x,v)$. Then $Y^{x,v}_{t}$
is the mean-square derivative of $X^{\gamma(h)}_{t}$ with respect
to initial condition in the following sense,
\begin{equation}
\lim_{h\to0}\frac{\mathbb{E}d^{2}_{g}(X^{\gamma(h)}_{t},\exp_{X^{\gamma(h)}_{t}}(hY^{x,v}_{t}))}{h^{2}}=0.\label{eq:mani-mean der}
\end{equation}
Moreover, for fixed $t$, the map $(x,v)\longmapsto Y^{x,v}_{t}$
is continuous in mean square.
\end{lem}

\begin{proof}
The statement is local. Consider first a relatively compact coordinate
chart and stop the processes before leaving the chart. After converting
the Stratonovich equations to Itô form, the coordinate representation
of \eqref{eq:manifold-prolongation} coincides with the stochastic
variational equation of the coordinate SDE. Proposition \ref{prop:msdiff}
therefore gives the corresponding mean-square differentiability statement
in Euclidean coordinates.

On a relatively compact chart, the Euclidean norm and the Riemannian
distance are locally equivalent. Moreover, 
\(
\gamma(h)=\exp_{x}(hv)+o(h)
\)
in any local coordinate representation. Hence the coordinate mean-square
estimate is equivalent to \eqref{eq:mani-mean der}. The result is
extended along a trajectory by successive localization in relatively
compact charts. Nonexplosion allows the stopping times to tend to
infinity. The continuity statement follows from the corresponding
continuous dependence estimate for the prolonged SDE. 
\end{proof}

The next lemma is the manifold analogue of the path-integration argument
used in Section~\ref{sec:main}.
\begin{lem}
\label{lem:mani-esti2}Let $\gamma:[0,1]\to M$ be a continuous and
piecewise $C^{1}$ curve in $M$. Let $Z^{\gamma(s),\gamma'(s)}_{t}$
be the solution of \eqref{eq:manifold-prolongation} with initial
condition $(\gamma(s),\gamma'(s))$. Then for every $t\ge0$, there
holds
\begin{equation}
\mathbb{E}d^{2}_{g}(X^{\gamma(0)}_{t},X^{\gamma(1)}_{t})\le\left(\int^{1}_{0}\sqrt{\mathbb{E}|Y^{\gamma(s),\gamma'(s)}_{t}|^{2}_{g}}ds\right)^{2}.\label{eq:manifold-path}
\end{equation}
\end{lem}
\begin{proof}
It is enough to consider $\gamma\in C^{1}$. Define 
\[
\rho(U,V):=\left(\E d^{2}_{g}(U,V)\right)^{1/2}.
\]
By Minkowski's inequality, $\rho$ is a metric on random variables
modulo almost-sure equality. For fixed $s$, Lemma~\ref{lem:mani-esti1}
and the triangle inequality give 
\begin{align*}
 & \rho\left(X^{\gamma(s+h)}_{t},X^{\gamma(s)}_{t}\right)\\
 & \quad\le\rho\left(X^{\gamma(s+h)}_{t},\exp_{X^{\gamma(s)}_{t}}\left(hY^{\gamma(s),\gamma'(s)}_{t}\right)\right)\\
 & \qquad+\rho\left(\exp_{X^{\gamma(s)}_{t}}\left(hY^{\gamma(s),\gamma'(s)}_{t}\right),X^{\gamma(s)}_{t}\right).
\end{align*}
The first term is $o(|h|)$. For the second term, the radial geodesic
gives 
\[
d_{g}\left(\exp_{X^{\gamma(s)}_{t}}(hY^{\gamma(s),\gamma'(s)}_{t}),X^{\gamma(s)}_{t}\right)\le|h|\,|Y^{\gamma(s),\gamma'(s)}_{t}|_{g}.
\]
Consequently, 
\begin{equation}
\limsup_{h\to0}\frac{\rho(X^{\gamma(s+h)}_{t},X^{\gamma(s)}_{t})}{|h|}\le\left(\E|Y^{\gamma(s),\gamma'(s)}_{t}|^{2}_{g}\right)^{1/2}.\label{eq:manifold-metric-derivative}
\end{equation}
The right-hand side is continuous in $s$. Applying Lemma~\ref{lem:metric_lemma}
to the curve $s\mapsto X^{\gamma(s)}_{t}$ in the metric space induced
by $\rho$ gives \eqref{eq:manifold-path}.
\end{proof}

\subsection{An intrinsic contraction criterion }\label{subsec:mani-no-coup}

We can now formulate the direct manifold counterpart of the Euclidean
differential contraction theorem.
\begin{thm}
\label{thm:stra-Rie}Suppose Assumption \ref{ass:manifold} holds.
If there exists a positive constant $\lambda$ such that
\begin{equation}
\mathcal{L}_{f_{0}}g+\frac{1}{2}\sum^{m}_{i=1}\mathcal{L}^{2}_{f_{i}}g\le-2\lambda g,\label{eq:thm-stra-ineq}
\end{equation}
then the system \eqref{eq:manifold-stratonovich} is $W_{2}$ contracting. 
\end{thm}

\begin{proof}
The infinitesimal generator of the prolonged system is
\[
\mathcal{B}=f^{\dagger}_{0}+\frac{1}{2}\sum^{m}_{i=1}\left(f^{\dagger}_{i}\right)^{2}.
\]
Let $h(v)=g(v,v)$, then by Lemma \ref{lem:lift},
\[
\mathcal{B}h=\mathcal{L}_{f_{0}}g+\frac{1}{2}\sum^{m}_{i=1}\mathcal{L}^{2}_{f_{i}}g\le-2\lambda h
\]
 Now, Dynkin's formula gives $\mathbb{E}h(Y^{x,y}_{t})\le e^{-\lambda t}h(Y^{x,y}_{0})$
or
\[
\mathbb{E}|Y^{x,y}_{t}|^{2}_{g}\le e^{-2\lambda t}|y|^{2}_{g}.
\]
Invoking Lemma \ref{lem:mani-esti2}, we get
\begin{align*}
\sqrt{\mathbb{E}d^{2}_{g}(X^{\gamma(0)}_{t},X^{\gamma(1)}_{t})} & \le\int^{1}_{0}\sqrt{e^{-2\lambda t}|\gamma'(s)|^{2}_{g}}ds\\
 & =e^{-\lambda t}\int^{1}_{0}|\gamma'(s)|_{g}ds
\end{align*}
for any continuous and piecewise $C^{1}$ curve. Since $\gamma$ is
arbitrary, it follows that
\[
\sqrt{\mathbb{E}d^{2}_{g}(X^{x}_{t},X^{y}_{t})}\le e^{-\lambda t}d_{g}(x,y).
\]
Applying the same technique as in the proof of Theorem \ref{thm:main}
(Step 2), we arrive at the desired conclusion.
\end{proof}

Below a simple example illustrating the use of Theorem \ref{thm:stra-Rie}.
\begin{example}[Brownian motion on sphere]
\label{exa:Sphere}The Brownian motion on sphere $S^{n-1}$ is the solution
to
\begin{equation}
dX_{t}=P_{X_{t}}\circ dB_{t}\label{eq:exa:-Sn}
\end{equation}
where $P_{x}=I-xx^{\top}$ and $B_{t}$ is a standard Brownian motion
in $\mathbb{R}^{n+1}$. Equip $S^{n-1}$ with the induced metric from
the ambiant space $\mathbb{R}^{n}$ and denote it as $g$. 

Let $\sigma_{i}(x)=P_{x}e_{i}=e_{i}-x_{i}x$, where $e_{i}$ is the
$i$-th coordinate vector in $\mathbb{R}^{n}$. Then it can be shown
that
\begin{align*}
\mathcal{L}^{2}_{\sigma_{i}}g & =(-2+6x^{2}_{i})g.
\end{align*}
Summing together, we arrive at 
\[
\frac{1}{2}\sum^{n}_{i=1}\mathcal{L}^{2}_{\sigma_{i}}g=-(n-3)g.
\]
Thus, the Brownian motion is contracting if $n\ge4$ with contraction
rate $\frac{n-3}{2}$. This estimate is conservative. As will be shown
in Section~\ref{sec:couple}, an appropriate differential coupling
recovers the curvature contribution and yields the sharper rate $\lambda=\frac{n-2}{2}$. 
\end{example}

\begin{example}[Killing diffusion]
\label{exa:killing}Suppose that in \eqref{eq:manifold-stratonovich},
$f_{i}$, $i=1,\cdots,m$ are Killing vector fields, i.e., $\mathcal{L}_{f_{i}}g=0$
for $i=1,\cdots,m$. The condition \eqref{eq:thm-stra-ineq} then
reduces to 
\[
\mathcal{L}_{f_{0}}g\le-2\lambda g.
\]
Consider for example the following system on $SO(3)$,
\begin{equation}
dR_{t}=\sum^{3}_{i=1}R_{t}E_{i}\circ dB^{i}_{t}\label{eq:SO(3)}
\end{equation}
where $\{E_{i}\}$ is the standard orthonormal basis of $\mathfrak{so}(3)$.
Let $g$ be a bi-invariant metric on $SO(3)$ defined as $g(V,W)=\frac{1}{2}{\rm tr}(V^{\top}W)$.
Then $RE_{i}$ are Killing fields on $SO(3)$. As a result, Theorem
\ref{thm:stra-Rie} cannot certify any strictly positive $W_{2}$
contraction rate for \eqref{eq:SO(3)}.
\end{example}

From Examples \ref{exa:Sphere} and \ref{exa:killing}, we see that
our methods are restrictive in analyzing $W_{2}$ contraction. The
limitation can be understood as follows. For system \eqref{eq:ItoSDE},
the measure $\mu_{t}$ depends only on $f$ and $A=\sigma\sigma^{\top}$.
That is, if we change $\sigma$ to $\delta$ such that $\delta\delta^{\top}=A$,
then $\mu_{t}$ stays unchanged. In the meantime, the inequality \eqref{eq:lifted-generator}
may change dramatically. 

\section{$W_{2}$ Contraction Beyond Synchronous Coupling}\label{sec:couple}

In view of the limitation of the methods developed in the previous
section, in this section, we introduce an additional infinitesimal
degree of freedom that accounts for the coupling of the driving Brownian
motions. We first develop the construction in Euclidean space, where
its probabilistic meaning is transparent, and then formulate its intrinsic
counterpart on Riemannian manifolds.

\subsection{Orthogonal differential coupling in Euclidean space}\label{subsec:Euclidean-coupling}

Consider a function $O:\mathbb{R}^{n}\times\mathbb{R}^{n}\to\mathbb{O}(m)$
where $\mathbb{O}(m)$ is the space of orthonomal matrices in $\mathbb{R}^{m}$,
i.e., $O(x,y)O(x,y)^{\top}=I_{m}$ for all $x,y\in\mathbb{R}^{n}$.
In addition, $O(x,x)=I$ for all $x\in\mathbb{R}^{n}$. We call such
$O(\cdot,\cdot)$ an \emph{orthonormal coupling}.

Let $v\in\mathbb{R}^{n}$ be a fixed vector. Consider an one-parameter
family of auxiliary systems of \eqref{eq:ItoSDE}
\begin{equation}
d\widetilde{X}^{x+hv}_{t}=f(t,\widetilde{X}^{x+hv}_{t})dt+\sigma(t,\widetilde{X}^{x+hv}_{t})O(X_{t},\widetilde{X}^{x+hv}_{t})dB_{t}\label{eq:aux_couple}
\end{equation}
for $h\in(-\epsilon,\epsilon)$, with initial condition $\widetilde{X}^{x+hv}_{0}=x+hv$.

Since $O$ is orthogonal, the law of $\widetilde{X}^{x+hv}_{t}$ is
nothing but $\mu^{x+hv}_{t}$, with $\mu^{x}_{t}$ being the law of
$X^{x}_{t}$. Indeed, since $O(X^{x}_{t},\widetilde{X}^{x+hv}_{t})$
is orthogonal, $\widetilde{B}_{t}=\int^{t}_{0}O(X_{s},\widetilde{X}^{x+hv}_{s})dB_{s}$
is again an $m$-dimensional Brownian motion.

Let us determine the infinitesimal effect of such a coupling. Assume
$O(\cdot,\cdot)$ is smooth. Then $h\mapsto\widetilde{X}^{x+hv}_{t}$
is mean square differentiable, and its derivative -- denoted as $\widetilde{Y}^{v}_{t}$
-- at $h=0$ satisfies
\begin{equation}
d\widetilde{Y}^{v}_{t}=\frac{\partial f}{\partial x}\widetilde{Y}^{v}_{t}dt+
\sum_{i=1}^m G_i(t,X_{t},\widetilde{Y}^{v}_{t})dB_{t}^i
\label{eq:Eu-O-couple-var}
\end{equation}
in which
\[
	G_i(t,x,y)=\frac{\partial\sigma_{i}(t,x)}{\partial x}y+ [\sigma(t,x)\Omega_{x}(y)]_i
\]
with $\Omega_{x}(y):=D_{2}O(x,x)(y)\in\mathfrak{so}(m)$ and the notation $[M]_i$ stands for the $i$-th column of the matrix $M$.

Compared with the synchronous variational dynamics, the only modification
is therefore 
\[
\frac{\partial\sigma_{i}}{\partial x}y\longrightarrow G_{i}(x,y)=\frac{\partial\sigma_{i}}{\partial x}y+[\sigma\Omega_{x}(y)]_{i}.
\]
The additional term represents the infinitesimal relative rotation
of the driving Brownian motions.

For a differential Lyapunov function $V(t,x,y)$, define the generator
associated with \eqref{eq:Eu-O-couple-var} by 
\begin{equation}
\begin{aligned}\widetilde{\mathcal{A}}_{\Omega}V={} & \partial_{t}V+(\nabla_{x}V)^{\top}f+(\nabla_{y}V)^{\top}\frac{\partial f}{\partial x}y\\
 & +\frac{1}{2}\sum^{m}_{i=1}\begin{bmatrix}\sigma_{i}\\
G_{i}
\end{bmatrix}^{\top}D^{2}_{(x,y)}V\begin{bmatrix}\sigma_{i}\\
G_{i}
\end{bmatrix}
\end{aligned}
\label{eq:lifted-gen-tilde}
\end{equation}

With these preparations, we are ready to state the following theorem.
\begin{thm}
\label{thm:O}Suppose that Assumption \ref{ass:regularity} holds.
Assume that there exist a mapping $(x,v)\mapsto\Omega_{x}(v)\in\mathfrak{so}(m)$,
linear in $v$, a function $V\in C([0,\infty)\times\R^{n}\times\R^{n})\cap C^{1,2,2}([0,\infty)\times\R^{n}\times(\R^{n}\backslash\{0\}))$
and constants $c_{1},c_{2},\lambda>0$ such that 
\begin{equation}
c_{1}|y|^{2}\le V(t,x,y)\le c_{2}|y|^{2}\label{eq:V-bounds-Omega}
\end{equation}
and 
\begin{equation}
\widetilde{A}_{\Omega}V(t,x,y)\le-2\lambda V(t,x,y)\label{eq:V-decay-Omega}
\end{equation}
for all $t\ge0$ and $x\in\R^{n}$, $y\in\mathbb{R}^{n}\backslash\{0\}$.
Then the system \eqref{eq:ItoSDE} is $W_{2}$ contracting. 
\end{thm}
\begin{proof}
By Dynkin's formula, with a localization procedure (Step 2 of the
proof of Theorem \ref{thm:main}), \eqref{eq:V-bounds-Omega} and
\eqref{eq:V-decay-Omega} together implies 
\[
\mathbb{E}|\widetilde{Y}^{v}_{t}|^{2}\le K^{2}e^{-2\lambda t}|v|^{2},\;\forall t\ge0.
\]

Since $\widetilde{Y}^{v}_{t}$ is the mean square derivative of $\widetilde{X}^{x+hv}_{t}$,
we deduce 
\[
\lim_{h\to0}\frac{\mathbb{E}|\widetilde{X}^{x+hv}_{t}-X^{x}_{t}|^{2}}{h^{2}}=\mathbb{E}|\widetilde{Y}^{v}_{t}|^{2}.
\]
Since $\widetilde{X}^{x+hv}_{t}\sim\mu^{x+hv}_{t}$, it follows that
\[
\limsup_{h\to0}\frac{W^{2}_{2}(\mu^{x+hv}_{t},\mu^{x}_{t})}{h^{2}}\le\mathbb{E}|\widetilde{Y}^{v}_{t}|^{2}\le K^{2}e^{-2\lambda t}|v|^{2}.
\]
Invoking Lemma \ref{lem:metric_lemma}, we arrive at the desired condition.
\end{proof}
\begin{rem}
The synchronous coupling result Theorem \ref{thm:main} is a special
case of Theorem \ref{thm:O} by taking $\Omega_{x}(v)=0$ or $O(x,y)=I$
for all $x,y$. 
\end{rem}
\begin{rem}
Although the finite coupling $O(x,z)$ provides a convenient probabilistic
interpretation, only its differential $\Omega_{x}(v)\in\mathfrak{so}(m)$
enters the contraction condition. The latter is therefore the essential
object in the subsequent development.
\end{rem}
The additional coupling freedom does not necessarily improve every
contraction condition. The geometric Brownian motion \eqref{eq:GBM}
provides a useful example. Let $V(x,y)=y^{\top}Py$. Then
\begin{align*}
\widetilde{A}_{\Omega}V(y) & =y^{\top}(A^{\top}P+PA)y\\
 & \qquad+\sum^{m}_{i=1}\left|C_{i}y+\sum^{m}_{j=1}C_{j}x(\Omega_{ji})\right|^{2}_{P}
\end{align*}
where $|z|_{P}=z^{\top}Pz$. At $x=0$, the coupling contribution
vanishes, and hence 
\[
\tilde{A}_{\Omega}V(y)=y^{\top}\left(A^{\top}P+PA+\sum^{m}_{i=1}C^{\top}_{i}PC_{i}\right)y.
\]
Consequently, the coupling cannot weaken the condition obtained in
the synchronous case. This is consistent with the necessity result
for \eqref{eq:GBM}. 

\subsection{Orthogonal coupling on Riemannian manifolds}\label{subsec:manifold-coupling}

We now extend the preceding construction to the Riemannian setting.
Consider the Stratonovich SDE \eqref{eq:manifold-stratonovich} on
a Riemannian manifold $(M,g)$.

A direct state-dependent orthogonal rotation of the diffusion vector
fields in the Stratonovich equation does not provide the desired coupling.
Indeed, if one formally writes 
\[
d\widetilde{X}_{t}=f(t,\widetilde{X}_{t})\,dt+\sigma(t,\widetilde{X}_{t})O(X_{t},\widetilde{X}_{t})\circ dB_{t},
\]
the state dependence of $O$ contributes additional drift terms when
the equation is converted to Itô form. Orthogonality therefore does
not, by itself, guarantee that the marginal law of $\widetilde{X}_{t}$
agrees with that of $X_{t}$.

However, we can turn the Stratonovich form into Itô form in local
coordinates and then apply the techniques in the previous section
for $V(t,x,v)=g_{x}(v,v)$. By doing so, we obtain the following 
geometric condition. 
\begin{thm}
\label{thm:coup-mani}Consider the system \eqref{eq:manifold-stratonovich}.
Suppose that Assumption \ref{ass:manifold} holds and that there exists
a mapping $\Omega_{x}:T_{x}M\to\mathfrak{so}(m)$, linear on $T_x M$, such that for some
$\lambda>0$,
\begin{equation}
g(\nabla_{v}\overline{f},v)+\frac{1}{2}\sum^{m}_{i=1}\left\{ -R(\sigma_{i},v,v,\sigma_{i})+|\eta_{i}|^{2}_{g}\right\} \le-\lambda|v|^{2}_{g}\label{eq:thm-coup-mani}
\end{equation}
for every $t,x\in M$ and $v\in T_{x}M$, where 
\[
\overline{f}=f+\frac{1}{2}\sum^{m}_{i=1}\nabla_{\sigma_{i}}\sigma_{i}
\]
 and 
\[
\eta_{i}(v)=\nabla_{v}\sigma_{i}+[\sigma\Omega_{x}(v)]_{i}.
\]
Then the system is $W_{2}$ contracting. 
\end{thm}

The proof of the theorem is provided in the Appendix.

\begin{rem}
For fixed $(x,v)$, the coupling freedom suggests the pointwise optimiza\-tion
problem 
\begin{equation}
\inf_{\Omega\in\mathfrak{so}(m)}\sum^{m}_{i=1}\left|\nabla_{v}\sigma_{i}+[\sigma\Omega]_{i}\right|^{2}_{g}.\label{eq:Omega-optimization}
\end{equation}
The ideal case is when one can choose $\Omega$ such that 
\[
\nabla_{v}\sigma_{i}+[\sigma\Omega]_{i}=0,\quad i=1,\ldots,m.
\]
In this situation, the variation of the diffusion frame is completely
removed by the coupling, and only the curvature contribution remains.
\end{rem}

\subsection{Ricci curvature effect and examples}\label{subsec:coupling-curvature}
\begin{cor}
\label{cor:Ricci}Consider the system \eqref{eq:manifold-stratonovich}
on $(M,g)$. Suppose that $m=n$ and that $\{\sigma_{i}\}^{n}_{i=1}$
forms an orthonormal frame on $M$. Then the system is $W_{2}$ contracting
if there exists a positive constant $\lambda$, such that
\[
g(\nabla_{v}\bar{f},v)-\frac{1}{2}{\rm Ric}(v,v)\le-\lambda|v|^{2}_{g}
\]
or every $x\in M$ and $v\in T_{x}M$, where $\bar{f}=f+\frac{1}{2}\sum^{m}_{i=1}\nabla_{\sigma_{i}}\sigma_{i}$. 
\end{cor}
\begin{proof}
For $v\in T_{x}M$, define $\omega_{ji}(x,v)$ as follows 
\[
\omega_{ji}(x,v)=g(\nabla_{v}\sigma_{i},\sigma_{j})
\]
Then $\omega_{ji}(x,v)=-\omega_{ij}(x,v)$ since $g(\sigma_{i},\sigma_{j})=\delta_{ij}$
and that
\[
g(\nabla_{v}\sigma_{i},\sigma_{j})+g(\sigma_{i},\nabla_{v}\sigma_{j})=0.
\]
Using this notation, 
\begin{align*}
\nabla_{v}\sigma_{i} & =\sum^{n}_{j=1}g(\nabla_{v}\sigma_{i},\sigma_{j})\sigma_{j}=\sum^{n}_{j=1}\omega_{ji}\sigma_{j}.
\end{align*}
Let $(\Omega_{x}(v)_{ij})=\omega_{ji}$. Then
\begin{align*}
\eta_{i} & =\nabla_{v}\sigma_{i}+[\sigma\Omega_{x}(v)]_{i}\\
 & =\sum^{n}_{j=1}\omega_{ji}\sigma_{j}+\sum^{n}_{j=1}\sigma_{j}(\Omega_{x}v)_{ji}\\
 & =\sum^{n}_{j=1}\omega_{ji}\sigma_{j}+\sum^{n}_{j=1}\sigma_{j}\omega_{ij}\\
 & =0.
\end{align*}
The conclusion now follows from Theorem \ref{thm:coup-mani}. 
\end{proof}

Corollary~\ref{cor:Ricci} exhibits directly the connection with
the classical Ricci/Bakry-Émery curvature criterion for Wasserstein
contraction \cite{bakry1985diffusions,vonrenesse2005transport,kuwada2010duality}.

We now revisit the two examples from Section~\ref{subsec:mani-no-coup}.
\begin{example}[Revisit of Examples \ref{exa:Sphere} ]
Consider the system \eqref{eq:exa:-Sn}. Recall
\(
\sigma_{i}=e_{i}-x_{i}x
\),
and we can compute
\[
\nabla_{v}\sigma_{i}=\sigma(x)\frac{\partial\sigma_{i}}{\partial x}v=-x_{i}v,
\]
where $v$ is a tangent vector at $x$, i.e., $x^{\top}v=0$. Hence
\begin{align*}
\nabla_{\sigma_{i}}\sigma_{i} & =-x_{i}(e_{i}-x_{i}x).
\end{align*}
Thus $\overline{f}=-\frac{1}{2}\sum^{n}_{i=1}\nabla_{\sigma_{i}}\sigma_{i}=\frac{1}{2}\sum^{n}_{i=1}x_{i}(e_{i}-x_{i}x)=0$. 

Choose $\Omega_{x}(v)=vx^{\top}-xv^{\top}$, we have 
\begin{align*}
\sigma\Omega_{x}(v) & =P_{x}\Omega_{x}(v)\\
 & =(I-xx^{\top})(vx^{\top}-xv^{\top})\\
 & =vx^{\top}-xv^{\top}+xv^{\top}\\
 & =vx^{\top}.
\end{align*}
Under this choice, $\eta_{i}=0$ in \eqref{eq:thm-coup-mani}. For
the curvature term, since $S^{n-1}$ has sectional curvature $1$,
we have
\begin{align*}
R(\sigma_{i},v,v,\sigma_{i}) & =|\sigma_{i}|^{2}|v|^{2}-(\sigma^{\top}_{i}v)^{2}\\
 & =(1-x^{2}_{i})|v|^{2}-v^{2}_{i}.
\end{align*}
Summing over $i$ gives 
\[
\sum^{n}_{i=1}R(\sigma_{i},v,v,\sigma_{i})=(n-2)|v|^{2}.
\]
Thus the Brownian motion on $S^{n-1}$ is $W_{2}$ contracting as
long as $n\ge3$ with contraction rate $\frac{n-2}{2}$. Thus orthogonal
coupling improves the synchronous estimate of Example~\ref{exa:Sphere}
from $(n-3)/2$ to $(n-2)/2$ and also recovers contraction in the
case $S^{2}$.
\end{example}
\begin{example}[Revisit of Example \ref{exa:killing}]
Consider again \eqref{eq:SO(3)}. Let $\sigma_{i}(R)=RE_{i}$ where
$E_{i}$ are the standard orthonormal basis of $\mathfrak{so}(3)$.
As a consequence $\{\sigma_{1},\sigma_{2},\sigma_{3}\}$ form an orthonormal
frame on $SO(3)$. Since $g$ is bi-invariant, we have
\[
\nabla_{\sigma_{i}}\sigma_{i}=\frac{1}{2}[\sigma_{i},\sigma_{i}]=0.
\]
Thus $\overline{f}=0$. It is known that the Ricci curvature of $SO(3)$
is $\frac{1}{2}g$. Hence the system is $W_{2}$ contractive with
rate $\frac{1}{4}$. 

This contrasts with Example~\ref{exa:killing}, where synchronous
coupling does not certify any positive contraction rate. %
\end{example}

\begin{example}[Diffusion on SPD]
\label{exa:SPD}
Let $\mathbb{S}^n_+$ be the space of real positive definite matrices
of dimension $n\times n$. 
Equip $\mathbb{S}^n_+$ with the affine-invariant metric
\[
    g_X(U,V)
    =
    \operatorname{tr}(X^{-1}UX^{-1}V).
\]
The manifold $\mathbb{S}^n_+$ is complete and has nonpositive
sectional curvature \cite{thanwerdas2023}.

Fix $X_\star\in \mathbb{S}^n_+$ and define
\[
    \Phi(X)
    :=
    \frac12 d_g^2(X,X_\star)
\]where $d_g$ is the Riemannian distance.
Let $\{\sigma_i\}_{i=1}^{d}$, $d=n(n+1)/2$, be a smooth orthonormal
frame on $\mathbb{S}^n_+$. Consider the system
\begin{equation}
    dX_t = (-k  \operatorname{grad}\Phi(X_t) + u) dt
    + \sum_{i=1}^{d} \sigma_i(X_t)\circ dB_t^i,
\label{eq:SPD-Stratonovich}
\end{equation}
where $k>0$ is a constant and $u$ is an input to be determined.

By taking $u= - \frac{1}{2}\sum_{i=1}^d \nabla_{\sigma_i}{\sigma_i}$, we have
\[
\begin{aligned}
    \bar f = -k \operatorname{grad}\Phi.
\end{aligned}
\]
Since $(\mathbb{S}^n_+,g)$ has nonpositive sectional curvature,
the squared distance is strongly geodesically convex, i.e., $ \operatorname{Hess}\Phi\ge g$.
Consequently,
\begin{equation}
    g(\nabla_v\bar f,v)
    =
    -k \operatorname{Hess}\Phi(v,v)
    \le
    -|v|_g^2.
    \label{eq:SPD-drift}
\end{equation}

For the affine-invariant metric, it can be shown that
\[
    \operatorname{Ric}_X(v,v)
    =
    -\frac n4
    \| v\|^2_{g_X} + \frac{1}{4}({\rm tr} (X^{-1}v))^2 \ge -\frac{n}{4}\|v\|_{g_X}^2.
\]
It follows that
\[
    g(\nabla_v\bar f,v)
    -
    \frac{1}{2} \operatorname{Ric}(v,v)
    \le
    -
    \left(
        k -\frac{n}{8}
    \right)|v|_g^2.
\]
Thus the system \eqref{eq:SPD-Stratonovich} is $W_2$ contracting if 
$k > \frac{n}{8}$.
\end{example}

In the next section, we use the differential coupling for a control
problem on $S^{2}$. There, a smooth deterministic feedback cannot
globally stabilize a prescribed target due to topological obstructions.
A suitably designed state-dependent diffusion, however, yields a globally
$W_{2}$-contracting closed-loop system whose invariant distribution
can be made arbitrarily concentrated around the target.

\section{Case Study: Noise-assisted Global Stabilization on a Compact Manifold}

In this section, we illustrate a control implication of Wasserstein contraction
on compact manifolds. Smooth deterministic feedback on a compact manifold
is generally subject to topological obstructions to global stabilization.
Noise provides a different possibility: rather than stabilizing to
a given point, one may stabilize the probability law globally to an
invariant distribution concentrated around the target.

We illustrate this phenomenon on the unit sphere $S^{2}$. Let $x_{\star}\in S^{2}$
denote the target and define 
\[
h(x):=x^{\top}_{\star}x,\quad P_{x}:=I-xx^{\top}.
\]
Consider first the deterministic closed-loop system 
\begin{equation}
\dot{x}=kP_{x}x_{\star},\label{eq:sphere-det-control}
\end{equation}
where $k$ is a positive constant. It can be easily shown that the
point $x_{\star}$ is asymptotically stable, but its antipodal point
$-x_{\star}$ is also an equilibrium. Hence $u=kP_{x}x_{\star}$ does
not globally stabilize $x_{\star}$.

We now add a state-dependent diffusion whose intensity is small near
the target and large near the antipodal point. 

For $\varepsilon>0$, consider the Stratonovich SDE 
\begin{equation}
\begin{aligned}dX_{t}={} & kP_{X_{t}}x_{\star}dt+a_{\varepsilon}(X_{t})P_{X_{t}}\circ dB_{t},\\
a_{\varepsilon}(x):={} & \sqrt{k\bigl(1+\varepsilon-h(x)\bigr)},
\end{aligned}
\label{eq:noise-assisted-sphere}
\end{equation}
where $B_{t}$ is a standard Brownian motion in $\mathbb{R}^{3}$.
Here the diffusion term $\int^{t}_{0}a_{\epsilon}P_{X_{t}}\circ dB_{t}$
can be seen as a random input. 

Notice that 
\[
a^{2}_{\varepsilon}(x_{\star})=k\varepsilon,\quad a^{2}_{\varepsilon}(-x_{\star})=k(2+\varepsilon).
\]
Thus the noise is weak near the desired state, while remaining strong
near the undesired equilibrium. We have the following result.
\begin{prop}[Noise-assisted global stabilization]
 \label{prop:noise-assisted-sphere} For each $\varepsilon>0$,
the stochast\-ic system \eqref{eq:noise-assisted-sphere} is globally
$W_{2}$ contracting. Moreover,

(i) If $\mu_{t}$ and $\nu_{t}$ are the laws generated from arbitrary
initial probability measures $\mu_{0}$ and $\nu_{0}$, then 
\begin{equation}
W_{2}(\mu_{t},\nu_{t})\le e^{-\lambda_{\varepsilon}t}W_{2}(\mu_{0},\nu_{0}),\quad\lambda_{\varepsilon}=\frac{k}{4}+\frac{k\varepsilon}{2}.\label{eq:noise-assisted-contraction}
\end{equation}

(ii) The system \eqref{eq:noise-assisted-sphere} admits a unique
invariant probability measure 
\begin{equation}
\pi_{\varepsilon}(dx)=\frac{1}{Z_{\varepsilon}}\bigl(\varepsilon+1-x^{\top}_{\star}x\bigr)^{-5/2}d{\rm vol}(x),\label{eq:noise-assisted-invariant}
\end{equation}
where $Z_{\varepsilon}$ is the normalization constant. 

(iii) As $\epsilon\to0$,
\begin{equation}
W_{2}(\pi_{\varepsilon},\delta_{x_{\star}})=O(\sqrt{\varepsilon}).\label{eq:pi-concentration}
\end{equation}
Consequently, the invariant distribution can be made arbitrarily concentrated
around $x_{\star}$ while preserving a strictly positive global contraction
rate.
\end{prop}
\begin{rem}
Proposition~\ref{prop:noise-assisted-sphere} separates convergence
and accuracy in a useful way. The parameter $k$ determines the time
scale of convergence, while $\varepsilon$ determines the amount of
noise remaining near the target. In particular, 
\[
\lim_{\varepsilon\to0}\lambda_{\varepsilon}=\frac{k}{4}>0,\quad\lim_{\varepsilon\to0}W_{2}(\pi_{\varepsilon},\delta_{x_{\star}})=0.
\]
Thus the invariant distribution can be made arbitrarily concentrated
around the target without causing the guaranteed global contraction
rate to vanish.
\end{rem}

The proof of the proposition is included in Appendix.

We simulate the system \eqref{eq:noise-assisted-sphere} choose $k=5$, $\epsilon=10^{-3}$, and $x_{*}=[0,0,1]^{\top}$.
The simulation interval is $[0,2]$. The results are shown in Fig. \ref{fig:sphere} and Fig.
\ref{fig:samplepaths}.

\begin{figure}[ht]
\begin{centering}
\includegraphics[scale=0.5]{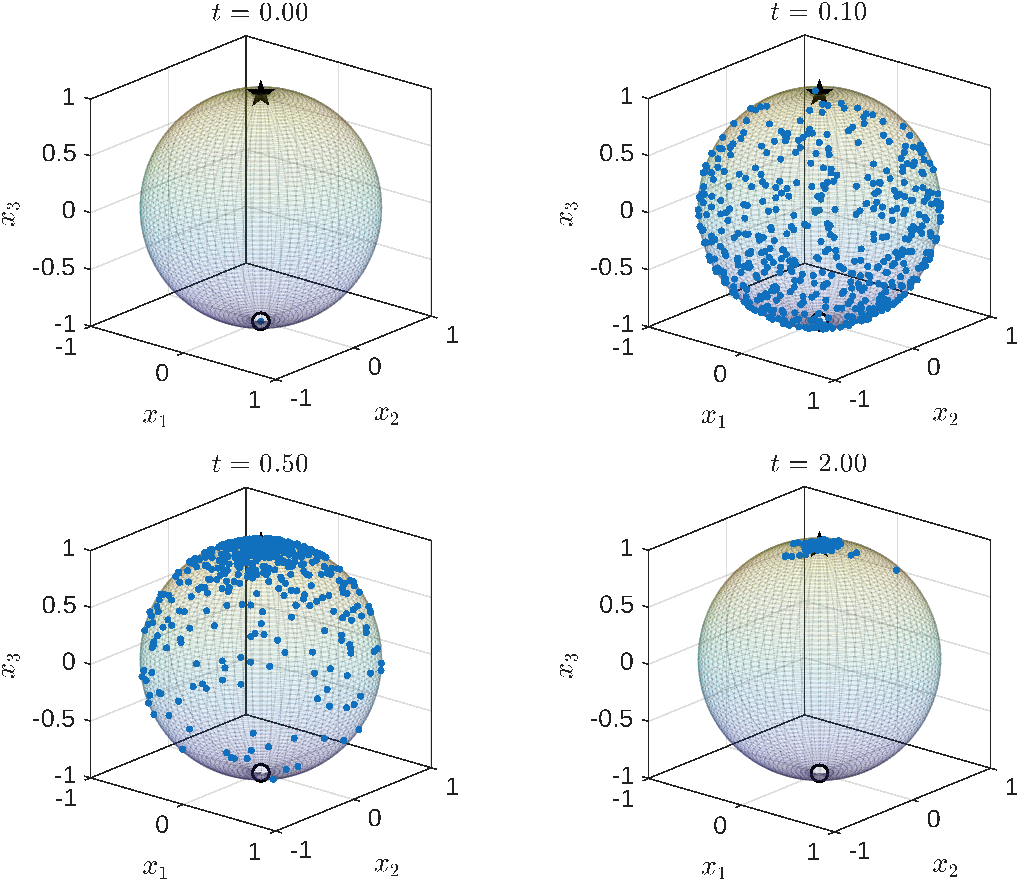}
\par\end{centering}
\caption{Evolution of the empirical probability distribution under the noise-assisted
feedback. All trajectories are initialized at the antipodal point
$-x_{\star}$. The distribution escapes from the undesired deterministic
equilibrium and rapidly concentrates near the target $x_{\star}$. \label{fig:sphere}}
\end{figure}

\begin{figure}[ht]
\begin{centering}
\includegraphics[scale=0.5]{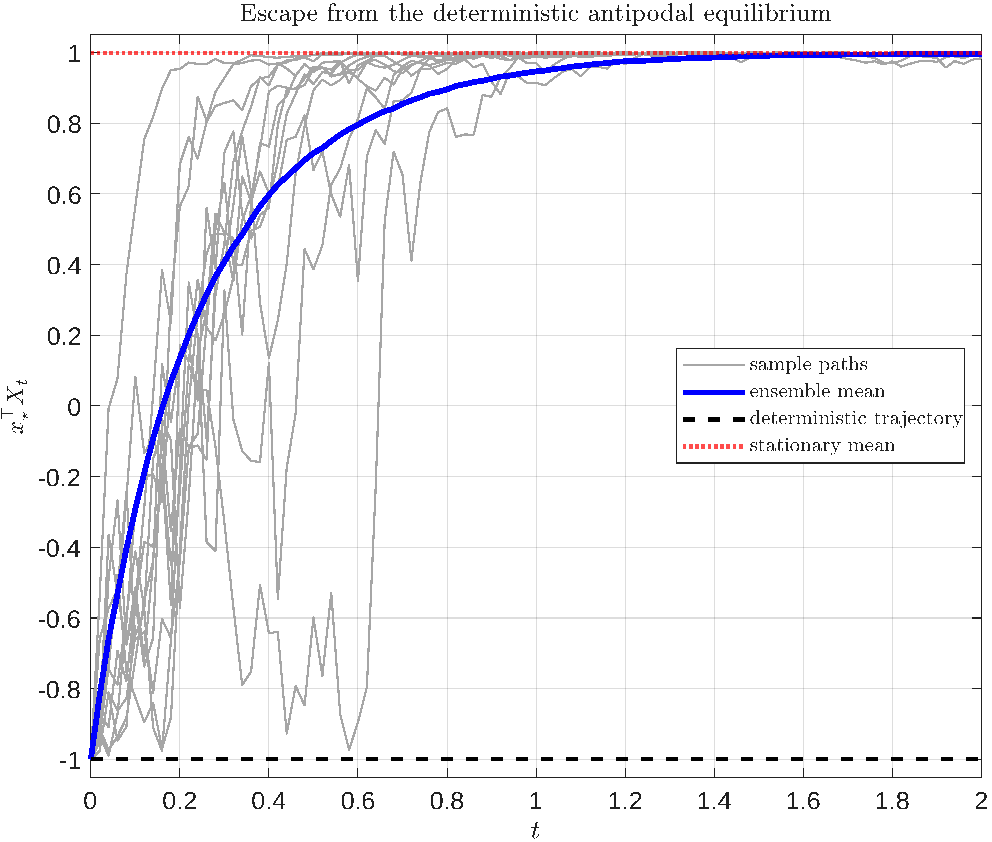}
\par\end{centering}
\caption{The deterministic trajectory initialized at $-x_{\star}$ remains
at $x^{\top}_{\star}x=-1$, whereas the stochastic trajectories escape
from the antipodal equilibrium. The ensemble mean approaches the stationary
mean associated with $\pi_{\varepsilon}$. \label{fig:samplepaths}}

\end{figure}

\section{Conclusion}


This paper developed a contraction-analysis approach to Wassers\-tein
contraction of stochastic systems. The main idea is to retain the
differential viewpoint underlying deterministic contraction theory.
In Euclidean
space, this yields criteria based on general Finsler-Lyapunov
functions, leading to both $W_1$ and $W_2$ contraction results. 

A feature that distinguishes the stochastic problem from its
deterministic counterpa\-rt is the freedom to couple the driving noises.
Synchronous coupling provides a natural starting point, but it can be
conservative. The orthogonal differential coupling
used in this paper incorporates this freedom directly into the
variational dynamics. On manifolds, the resulting criterion 
recovers the familiar curvature-based
contraction condition in the orthonormal-frame case. The examples on
the sphere and on Lie groups illustrate that this additional coupling
freedom can reveal contraction that is hidden by synchronous coupling.

The control example further suggests that stochasticity need not be
viewed only as a disturbance. By shaping the diffusion as
a control input, noise can be used to modify the global behavior of
the probability law. 

The results also leave several directions open. For example, 
how to use Wasserstein contraction results
for observer design, in which case, non-synchronous noise has to be taken
into consideration.


\section*{Acknowledgement}

The author would like to thank Sunghyun Koo for inspiring discussions.
Large language models were used during the preparation of this manuscript
to assist with wording, language editing, and proofreading.

\appendix

\section{Appendix}

\label{app:msdiff} 

\subsection{Some Riemannian geometry}

Let $(M,g)$ be a Riemannian metric, and $\nabla$ the Levi-Civita
connection. Let $X,Y,Z,W$ be vector fields on $M$. The curvature
tensor is defined as
\[
R(X,Y)Z=\nabla_{X}\nabla_{Y}Z-\nabla_{Y}\nabla_{X}Z-\nabla_{[X,Y]}Z.
\]
We also denote 
\[
R(X,Y,Z,W)=g(R(X,Y)Z,W).
\]
The Ricci tensor is defined as
\[
{\rm Ric}(V,W)=\sum^{n}_{i=1}R(E_{i},W,V,E_{i})
\]
where $\{E_{i}\}$ is an orthonormal basis. 

Let $u,v\in T_{x}M$. The metric Lie derivative $\mathcal{L}_{X}g$
at $x$ is a defined as 
\begin{equation} \label{eq:metric-Lie}
(\mathcal{L}_{X}g)(u,v)=\left.\frac{d}{dt}\right|_{t=0}g(D\phi_{t}(u),D\phi_{t}(v))
\end{equation}
where $\phi_{t}$ is the flow of the vector field $X$. $X$ is called
a Killing field if $\mathcal{L}_{X}g=0$. 

For a scalar valued function $\varphi:M\to\mathbb{R}$, denote $\nabla\varphi$
as the gradient of $\varphi$, i.e., it is the unique vector such
that $g(\nabla\varphi,v)=d\varphi(v)$ for any $v\in TM$. Denote
${\rm Hess}\,\varphi$ as the Hessian of $\varphi$, which is defined
as 
\[
{\rm Hess}\,\varphi(v,w)=g(\nabla_{v}\nabla\varphi,w).
\]

\subsection{A Lemma on Metric Derivative}
\begin{lem}
\label{lem:metric_lemma}Let $(X,d)$ be a metric space and $\gamma:[0,1]\to X$
be continuous. Suppose that $v:[0,1]\to X$ is continous and
\[
\limsup_{h\to0}\frac{d(\gamma(s+h),\gamma(s))}{|h|}\le v(s)
\]
for every $s\in(0,1)$. Then
\[
d(\gamma(t),\gamma(s))\le\int^{t}_{s}v(r)dr,\quad0\le s\le t\le1.
\]
\end{lem}

\subsection{Proof of Proposition \ref{prop:msdiff}}

Fix $x,y\in\R^{n}$ and define the difference quotient 
\begin{equation}
Z^{h}_{t}:=\frac{X^{x+hy}_{t}-X^{x}_{t}}{h}.
\end{equation}
By the mean-value theorem, 
\begin{equation}
dZ^{h}_{t}=A^{h}_{t}Z^{h}_{t}\,dt+\sum^{m}_{i=1}G^{h}_{i,t}Z^{h}_{t}\,dB^{i}_{t},\quad Z^{h}_{0}=y,\label{eq:Zh}
\end{equation}
where 
\begin{align}
A^{h}_{t} & :=\int^{1}_{0}D_{x}f\!\left(t,X^{x}_{t}+r(X^{x+hy}_{t}-X^{x}_{t})\right)dr,\label{eq:Ah}\\
G^{h}_{i,t} & :=\int^{1}_{0}D_{x}\sigma_{i}\!\left(t,X^{x}_{t}+r(X^{x+hy}_{t}-X^{x}_{t})\right)dr.\label{eq:Gh}
\end{align}
Let $Y^{x,y}$ denote the solution of \eqref{eq:Ito_SDE_var}. Since
the coefficients of that linear SDE are bounded on $[0,T]$, 
\begin{equation}
\E\sup_{0\le t\le T}|Y^{x,y}_{t}|^{2}<\infty.\label{eq:Ybound}
\end{equation}
Set $R^{h}_{t}:=Z^{h}_{t}-Y^{x,y}_{t}$. Subtracting the two equations
gives 
\begin{align}
dR^{h}_{t}={} & A^{h}_{t}R^{h}_{t}\,dt+\sum^{m}_{i=1}G^{h}_{i,t}R^{h}_{t}\,dB^{i}_{t}\nonumber \\
 & +(A^{h}_{t}-A_{t})Y^{x,y}_{t}\,dt+\sum^{m}_{i=1}(G^{h}_{i,t}-G_{i,t})Y^{x,y}_{t}\,dB^{i}_{t},\label{eq:Rh}
\end{align}
where $A_{t}=A(t,X^{x}_{t})$ and $G_{i,t}=G_{i}(t,X^{x}_{t})$. Standard
Burkholder--Davis--Gundy and Grönwall estimates for linear SDEs
give 
\begin{align}
\E\sup_{0\le t\le T}|R^{h}_{t}|^{2}\le{} & C_{T}\E\int^{T}_{0}|(A^{h}_{t}-A_{t})Y^{x,y}_{t}|^{2}\,dt\nonumber \\
 & +C_{T}\E\int^{T}_{0}\sum_{i}|(G^{h}_{i,t}-G_{i,t})Y^{x,y}_{t}|^{2}\,dt.\label{eq:Rbound}
\end{align}
On the other hand, the usual continuous-dependence estimate for \eqref{eq:ItoSDE}
yields 
\begin{equation}
\E\sup_{0\le t\le T}|X^{x+hy}_{t}-X^{x}_{t}|^{2}\le C_{T}h^{2}|y|^{2}.\label{eq:Xcontinuous}
\end{equation}
Hence $X^{x+hy}_{t}\to X^{x}_{t}$ in probability uniformly on $[0,T]$.
By the continuity of $D_{x}f$ and $D_{x}\sigma_{i}$, 
\begin{equation}
A^{h}_{t}\to A_{t},\quad G^{h}_{i,t}\to G_{i,t}
\end{equation}
in probability for almost every $t$. The derivative bound \eqref{eq:derivative-bound},
together with \eqref{eq:Ybound}, makes the integrands on the right-hand
side of \eqref{eq:Rbound} uniformly integrable. Vitali's theorem
and dominated convergence in time therefore imply that the right-hand
side of \eqref{eq:Rbound} tends to zero. This proves \eqref{eq:msdiff}.

Finally, applying the same continuous-dependence estimate to two copies
of the linear variational equation \eqref{eq:Ito_SDE_var}, together
with continuous dependence of $X^{x}$ on $x$, proves continuity
of $(x,y)\mapsto Y^{x,y}$ in $L^{2}(\Omega;C([0,T];\R^{n}))$. \hfill{}$\square$

\subsection{Proof of Theorem \ref{thm:coup-mani}}

Since Stratonovich calculus obeys the ordinary chain rule, in local
coordinates, \eqref{eq:manifold-stratonovich} takes the form
\begin{align}
dX_{t} & =\tilde{f}dt+\sum_{i}\sigma_{i}dB^{i}_{t}\label{eq:local-stra}
\end{align}
where $\tilde{f}=f+\frac{1}{2}\sum\frac{\partial\sigma_{i}}{\partial x}\sigma_{i}$.%

Now, we try to apply Theorem \ref{thm:O} to this system for $V(t,x,v)=g(v,v)=g_{ij}(x)v^{i}v^{j}$.
We verify that, in local coordinates, $\widetilde{A}_{\Omega}$ acting
on $V$ along the system \eqref{eq:local-stra} gives exactly
\begin{equation}
2g(\nabla_{v}\bar{f},v)+\sum^{m}_{i=1}\left(-R(\sigma_{i},v,v,\sigma_{i})+|\eta_{i}(v)|^{2}_{g}\right),\label{eq:goal_intrinsic}
\end{equation}
where 
\[
\bar{f}=f+\frac{1}{2}\sum^{m}_{i=1}\nabla_{\sigma_{i}}\sigma_{i},\quad\eta_{i}(v)=\nabla_{v}\sigma_{i}+[\sigma\Omega_{x}(v)]_{i}.
\]

Fix $x\in M$ and choose Riemannian normal coordinates $(x^{1},\ldots,x^{n})$
centered at $x$. At $x$, 
\[
g_{ab}=\delta_{ab},\quad\partial_{r}g_{ab}=0,\quad\Gamma^{a}_{bc}=0.
\]
At the center of the normal coordinates we have 
\begin{align*}
\partial_{x^{r}}V & =0,\quad\partial_{v^{a}}V=2v^{a},\quad\partial^{2}_{x^{r}v^{a}}V=0,\\
\partial^{2}_{v^{a}v^{b}}V & =2\delta_{ab},\quad\partial^{2}_{x^{p}x^{q}}V=(\partial_{pq}g_{ab})v^{a}v^{b}.
\end{align*}
Hence, by the definition of $\widetilde{\mathcal{A}}_{\Omega}$, 
\begin{equation}
\widetilde{\mathcal{A}}_{\Omega}V=2v^{a}v^{r}\partial_{r}\tilde{f}^{a}+\frac{1}{2}\sum^{m}_{i=1}\sigma_{ip}\sigma_{iq}(\partial_{pq}g_{ab})v^{a}v^{b}+\sum^{m}_{i=1}|G_{i}|^{2}.\label{eq:normal-generator}
\end{equation}
where
\[
G^{a}_{i}=(\partial_{r}\sigma_{ia})v^{r}+[\sigma\Omega_{x}(v)]^{a}_{i}.
\]

Now, we expand \eqref{eq:goal_intrinsic} in the same local coordinates.
First, we compute $\eta_{i}$. At $x$, since $\Gamma^{a}_{bc}=0$,
\[
(\nabla_{v}\sigma_{i})^{a}=v^{r}\partial_{r}\sigma_{ia},
\]
and therefore 
\begin{equation}
\eta^{a}_{i}(v)=v^{r}\partial_{r}\sigma_{ia}+[\sigma\Omega_{x}(v)]^{a}_{i}\label{eq:eta-G_i}
\end{equation}
which is exactly $G^{a}_{i}$. 

Next, we evaluate $g(\nabla_{v}\bar{f},v)$. The coordinate components
of $\bar{f}$ are 
\[
\bar{f}^{a}=\tilde{f}^{a}+\frac{1}{2}\sum^{m}_{i=1}\Gamma^{a}_{pq}\sigma_{ip}\sigma_{iq}.
\]
Differentiating and evaluating at $x$ yields 
\[
\partial_{r}\bar{f}^{a}=\partial_{r}\tilde{f}^{a}+\frac{1}{2}\sum^{m}_{i=1}(\partial_{r}\Gamma^{a}_{pq})\sigma_{ip}\sigma_{iq}.
\]
Since $\Gamma(x)=0$, 
\[
g(\nabla_{v}\bar{f},v)=v^{a}v^{r}\partial_{r}\bar{f}^{a},
\]
and hence 
\begin{equation}
2g(\nabla_{v}\bar{f},v)=2v^{a}v^{r}\partial_{r}\tilde{f}^{a}+\sum^{m}_{i=1}v^{a}v^{r}(\partial_{r}\Gamma^{a}_{pq})\sigma_{ip}\sigma_{iq}.\label{eq:drift-normal}
\end{equation}

It remains to expand the curvature term. By metric compatibility,
\[
\partial_{p}g_{ab}=g_{cb}\Gamma^{c}_{pa}+g_{ac}\Gamma^{c}_{pb}.
\]
Differentiating and evaluating at $x$ gives 
\[
\partial_{pq}g_{ab}=\partial_{p}\Gamma^{b}_{qa}+\partial_{p}\Gamma^{a}_{qb}.
\]
Therefore, using the symmetry in $a,b$, 
\[
\frac{1}{2}\sigma_{ip}\sigma_{iq}(\partial_{pq}g_{ab})v^{a}v^{b}=\sigma_{ip}\sigma_{iq}v^{a}v^{b}\partial_{p}\Gamma^{a}_{qb}.
\]
On the other hand, since $\Gamma(x)=0$, 
\[
R^{a}{}_{qpb}=\partial_{p}\Gamma^{a}_{bq}-\partial_{b}\Gamma^{a}_{pq},
\]
so that, using $\Gamma^{a}_{bq}=\Gamma^{a}_{qb}$, 
\begin{equation}
\begin{aligned}-R & (\sigma_{i},v,v,\sigma_{i})=\sigma_{ip}\sigma_{iq}v^{a}v^{b}\left(\partial_{p}\Gamma^{a}_{qb}-\partial_{b}\Gamma^{a}_{pq}\right)\\
 & =\frac{1}{2}\sigma_{ip}\sigma_{iq}(\partial_{pq}g_{ab})v^{a}v^{b}-v^{a}v^{r}(\partial_{r}\Gamma^{a}_{pq})\sigma_{ip}\sigma_{iq}.
\end{aligned}
\label{eq:R-coord}
\end{equation}
Combining \eqref{eq:eta-G_i}, \eqref{eq:drift-normal} and \eqref{eq:R-coord}
together, we get
\begin{align*}
 & 2g(\nabla_{v}\bar{f},v)+\sum^{m}_{i=1}\left(-R(\sigma_{i},v,v,\sigma_{i})+|\eta_{i}(v)|^{2}_{g}\right)\\
= & 2v^{a}v^{r}\partial_{r}\tilde{f}^{a}+\sum^{m}_{i=1}v^{a}v^{r}(\partial_{r}\Gamma^{a}_{pq})\sigma^{p}_{i}\sigma^{q}_{i}+\sum^{m}_{i=1}|G_{i}|^{2}\\
 & \qquad-\sum^{m}_{i=1}R(\sigma_{i},v,v,\sigma_{i})\\
= & 2v^{a}v^{r}\partial_{r}\tilde{f}^{a}+\sum^{m}_{i=1}v^{a}v^{r}(\partial_{r}\Gamma^{a}_{pq})\sigma_{ip}\sigma_{iq}+\sum^{m}_{i=1}|G_{i}|^{2}\\
 & +\frac{1}{2}\sum_{i}\sigma_{ip}\sigma_{iq}(\partial_{pq}g_{ab})v^{a}v^{b}-\sum_{i}v^{a}v^{r}(\partial_{r}\Gamma^{a}_{pq})\sigma_{ip}\sigma_{iq}\\
= & 2v^{a}v^{r}\partial_{r}\tilde{f}^{a}+\sum^{m}_{i=1}|G_{i}|^{2}+\frac{1}{2}\sum_{i}\sigma_{ip}\sigma_{iq}(\partial_{pq}g_{ab})v^{a}v^{b}\\
= & \widetilde{A}_{\Omega}g(v,v),
\end{align*}
as desired.

Now, invoking Theorem \eqref{thm:O}, in local coordinates we have
$W_{2}$ contraction. However, the state may leave a coordinate chart
in finite time. We now globalize the local estimate. 

Let 
\(
(\Omega,\mathcal{F},\{\mathcal{F}_{t}\}_{t\ge0},\mathbb{P})
\)
be the filtered probability space on which the Brownian motion $B_{t}$
and the solutions are defined. We assume that the filtration satisfies
the usual conditions and that $B_{t}$ is an $\{\mathcal{F}_{t}\}$-Brownian
motion. Here $\mathcal{F}_{t}$ denotes the information available
up to time $t$; for instance, for the augmented natural filtration
of $B$, $\mathcal{F}_{t}=\sigma(B_{s}:0\le s\le t)\vee\mathcal{N}$,
where $\mathcal{N}$ denotes the collection of $\mathbb{P}$-null
sets.

Let $\{(U_{\alpha},\varphi_{\alpha})\}$ be an atlas of $M$. By refining
the atlas if necessary, choose relatively compact open sets \footnote{The notation $W_{\alpha}\Subset U_{\alpha}$ means $\overline{W_{\alpha}}\Subset U_{\alpha}$.}
\[
W_{\alpha}\Subset U_{\alpha},\quad\bigcup_{\alpha}W_{\alpha}=M.
\]
Let $V_{\alpha}(x,v)=g^{(\alpha)}_{ab}v^{a}v^{b}$. By the preceding
local-coordinate calculation, 
\[
\widetilde{A}^{(\alpha)}_{\Omega}V_{\alpha}=2g(\nabla_{v}\bar{f},v)+\sum^{m}_{i=1}[-R(\sigma_{i},v,v,\sigma_{i})+|\eta_{i}(v)|^{2}_{g}].
\]
Hence \eqref{eq:V-decay-Omega} implies 
\(
\widetilde{A}^{(\alpha)}_{\Omega}V_{\alpha}\le-2\lambda V_{\alpha}.
\)

We now concatenate the corresponding local estimates. Fix a compact
set $K\subset M$ containing the initial point. Since $K$ is compact,
there exist finitely many indices $\alpha_{1},\ldots,\alpha_{N}$
such that 
\[
K\subset\bigcup^{N}_{j=1}W_{\alpha_{j}},\quad W_{\alpha_{j}}\Subset U_{\alpha_{j}}.
\]
Define the measurable chart-selection map 
\[
\kappa(x):=\min\{j\in\{1,\ldots,N\}:x\in W_{\alpha_{j}}\},\quad x\in K.
\]
Let 
\[
\zeta_{K}:=\inf\{t\ge0:X_{t}\notin K\}
\]
be the exit time from $K$. We define the chart-switching times recursively.
Set 
\[
\tau_{0}:=0,\quad J_{0}:=\kappa(X_{0}),
\]
and, on the event $\{\tau_{k}<\zeta_{K}\}$, let 
\(
J_{k}:=\kappa(X_{\tau_{k}})
\),
and 
\begin{equation}
\tau_{k+1}:=\left(\inf\{t\ge\tau_{k}:X_{t}\notin U_{\alpha_{J_{k}}}\}\right)\wedge\zeta_{K}.\label{eq:chart-switch-time}
\end{equation}
If $\tau_{k}=\zeta_{K}$, we set $\tau_{k+1}=\zeta_{K}$. Thus, on
$[\tau_{k},\tau_{k+1})$, the process is represented in the chart
$U_{\alpha_{J_{k}}}$. Moreover, if $\tau_{k+1}<\zeta_{K}$, then
$X_{\tau_{k+1}}\in K$ and hence belongs to at least one of the sets
$W_{\alpha_{j}}$, so that $J_{k+1}$ is well defined.

At the stopping time $\tau_{k}$, the local coordinate initial conditions
are in general random. They are, however, $\mathcal{F}_{\tau_{k}}$-measurable,
where 
\[
\mathcal{F}_{\tau_{k}}:=\left\{ A\in\mathcal{F}:A\cap\{\tau_{k}\le t\}\in\mathcal{F}_{t}\text{ for every }t\ge0\right\} .
\]
Consider the shifted Brownian motion 
\[
B^{(k)}_{s}:=B_{\tau_{k}+s}-B_{\tau_{k}},\quad s\ge0.
\]
By the strong Markov property of Brownian motion, $B^{(k)}$ is, conditionally
on $\mathcal{F}_{\tau_{k}}$, a Brownian motion independent of $\mathcal{F}_{\tau_{k}}$.
Hence, conditionally on $\mathcal{F}_{\tau_{k}}$, the local dynamics
on $[\tau_{k},\tau_{k+1}]$ are precisely the Euclidean dynamics considered
previously, with the values at $\tau_{k}$ regarded as fixed initial
conditions.

Since the Euclidean exponential estimate holds for every deterministic
initial condition, it can therefore be applied conditionally at $\tau_{k}$.
Denote by $Z^{(k)}_{t}=\varphi_{\alpha_{J_{k}}}(X_{t})$ the local
coordinates on this interval and by $Y^{(k)}_{t}$ the corresponding
local variational variable (mean square derivative). Then, for every
$s\ge0$, {\small
\begin{equation}
\begin{aligned} & \mathbb{E}\Big[e^{2\lambda((\tau_{k}+s)\wedge\tau_{k+1}-\tau_{k})}V_{\alpha_{J_{k}}}\Big(Z^{(k)}_{(\tau_{k}+s)\wedge\tau_{k+1}},Y^{(k)}_{(\tau_{k}+s)\wedge\tau_{k+1}}\Big)\,\Big|\,\mathcal{F}_{\tau_{k}}\Big]\\
 & \hspace{4cm}\le V_{\alpha_{J_{k}}}\big(Z^{(k)}_{\tau_{k}},Y^{(k)}_{\tau_{k}}\big).
\end{aligned}
\label{eq:chart-esti}
\end{equation}
}

We next observe that the quantity appearing in \eqref{eq:chart-esti}
is unchanged when coordinates are switched. Suppose that $X_{\tau_{k+1}}\in U_{\alpha_{J_{k}}}\cap U_{\alpha_{J_{k+1}}}$
and let $\psi=\varphi_{\alpha_{J_{k+1}}}\circ\varphi^{-1}_{\alpha_{J_{k}}}$
be the transition map. The local tangent coordinates satisfy 
\[
Y^{(k+1)}_{\tau_{k+1}}=D\psi\big(Z^{(k)}_{\tau_{k+1}}\big)Y^{(k)}_{\tau_{k+1}}.
\]
Since the components of the Riemannian metric transform tensorially,
\begin{equation}
V_{\alpha_{J_{k+1}}}\big(Z^{(k+1)}_{\tau_{k+1}},Y^{(k+1)}_{\tau_{k+1}}\big)=V_{\alpha_{J_{k}}}\big(Z^{(k)}_{\tau_{k+1}},Y^{(k)}_{\tau_{k+1}}\big).\label{eq:Vk+1=00003DVk}
\end{equation}
Thus no chart-dependent multiplicative constants are introduced when
passing from one coordinate chart to another.

It remains to note that only finitely many chart switches can occur
before $\zeta_{K}$ on every finite time interval. Since the cover
is finite and $W_{\alpha_{j}}\Subset U_{\alpha_{j}}$, we have 
\[
\delta:=\min_{1\le j\le N}d_{g}\big(\overline{W_{\alpha_{j}}},M\setminus U_{\alpha_{j}}\big)>0.
\]
At time $\tau_{k}$, the process lies in $W_{\alpha_{J_{k}}}$, while
a genuine chart switch at $\tau_{k+1}<\zeta_{K}$ occurs only after
the process has exited $U_{\alpha_{J_{k}}}$. Hence 
\(
d_{g}(X_{\tau_{k}},X_{\tau_{k+1}})\ge\delta.
\)
Since the sample paths are continuous, they are uniformly continuous
on every compact time interval. Therefore infinitely many such switches
cannot occur in finite time.

Consequently, for every fixed $t$, the stopped path up to $t\wedge\zeta_{K}$
is covered by finitely many coordinate intervals $[\tau_{k},\tau_{k+1}]$.
Applying \eqref{eq:chart-esti} successively on these intervals, using
\eqref{eq:Vk+1=00003DVk} at each switching time and the tower property
of conditional expectation, yields 
\begin{equation}
\mathbb{E}\left[e^{2\lambda(t\wedge\zeta_{K})}\mathcal{V}_{t\wedge\zeta_{K}}\right]\le\mathcal{V}_{0},\label{eq:esti-com-V}
\end{equation}
where $\mathcal{V}_{t}$ denotes the common value $\mathcal{V}_{t}=V_{\alpha}\big(Z^{(\alpha)}_{t},Y^{(\alpha)}_{t}\big)$
in any chart containing $X_{t}$. By \eqref{eq:Vk+1=00003DVk}, this
value is independent of the chosen chart.

Finally, choose a compact exhaustion 
\[
K_{1}\Subset K_{2}\Subset\cdots,\quad\bigcup^{\infty}_{n=1}K_{n}=M,
\]
with $X_{0}\in K_{1}$, and define 
\[
\zeta_{n}:=\inf\{t\ge0:X_{t}\notin K_{n}\}.
\]
Under the nonexplosion assumption, $\zeta_{n}\to \infty$ {almost surely}.
For every fixed $t$, 
\(
t\wedge\zeta_{n} \to t
\) {almost surely}.
Since $\mathcal{V}_{t}\ge0$, Fatou's lemma and \eqref{eq:esti-com-V}
give 
\[
\begin{aligned}\mathbb{E}\left[e^{2\lambda t}\mathcal{V}_{t}\right] & \le\liminf_{n\to\infty}\mathbb{E}\left[e^{2\lambda(t\wedge\zeta_{n})}\mathcal{V}_{t\wedge\zeta_{n}}\right]\\
 & \le\mathcal{V}_{0}.
\end{aligned}
\]
for every $t\ge0$. From this we deduce 
\[
\mathbb{E}|Y_{t}|^{2}_{g}\le e^{-2\lambda t}\mathbb{E}|Y_{0}|^{2}.
\]
Now, proceeding as in the proof of Theorem \ref{thm:stra-Rie} we
can show
\[
\sqrt{\mathbb{E}d^{2}_{g}(X^{x}_{t},X^{y}_{t})}\le e^{-\lambda t}d_{g}(x,y).
\]

\subsection{Proof of Proposition \ref{prop:noise-assisted-sphere}}

(i) Let 
\[
s_{i}(x):=P_{x}e_{i},\quad\sigma_{i}(x):=a_{\varepsilon}(x)s_{i}(x),\quad i=1,2,3,
\]
where $\{e_{i}\}^{3}_{i=1}$ is the standard basis of $\mathbb{R}^{3}$.
We first compute the covariant Itô drift appearing in Theorem~\ref{thm:coup-mani}.
On $S^{2}$, 
\[
\nabla_{v}s_{i}=-x_{i}v,\quad\sum^{3}_{i=1}\nabla_{s_{i}}s_{i}=0.
\]
It follows that 
\begin{align*}
\sum^{3}_{i=1}\nabla_{\sigma_{i}}\sigma_{i} & =a_{\varepsilon}\sum^{3}_{i=1}s_{i}(a_{\varepsilon})s_{i}+a^{2}_{\varepsilon}\sum^{3}_{i=1}\nabla_{s_{i}}s_{i}\\
 & =a_{\varepsilon}\nabla a_{\varepsilon}=\frac{1}{2}\nabla a^{2}_{\varepsilon}.
\end{align*}
Since $a^{2}_{\varepsilon}=k(1+\varepsilon-h)$, we obtain 
\[
\sum^{3}_{i=1}\nabla_{\sigma_{i}}\sigma_{i}=-\frac{k}{2}\nabla h.
\]
Hence 
\begin{equation}
\overline{f}=k\nabla h+\frac{1}{2}\sum^{3}_{i=1}\nabla_{\sigma_{i}}\sigma_{i}=\frac{3k}{4}\nabla h.\label{eq:noise-assisted-bar-f}
\end{equation}

We next construct the differential coupling. For $v\in T_{x}S^{2}$,
take $\Omega_{x}(v)=vx^{\top}-xv^{\top}$. Since $P_{x}\Omega_{x}(v)=vx^{\top}$,
we have 
\[
[\sigma\Omega_{x}(v)]_{i}=a_{\varepsilon}x_{i}v.
\]
On the other hand, 
\[
\nabla_{v}\sigma_{i}=a_{\varepsilon}'(h)dh(v)s_{i}-a_{\varepsilon}x_{i}v.
\]
Therefore the second term is canceled by the orthogonal differential
coupling, and 
\begin{equation}
\eta_{i}(v)=\nabla_{v}\sigma_{i}+[\sigma\Omega_{x}(v)]_{i}=a_{\varepsilon}'(h)\,dh(v)s_{i}.\label{eq:noise-assisted-eta}
\end{equation}
Since $\sum^{3}_{i=1}|s_{i}|^{2}={\rm tr}(P_{x})=2$, and
\[
|dh(v)|^{2}\le|\nabla h|^{2}|v|^{2}=(1-h^{2})|v|^{2},
\]
\[
\bigl(a_{\varepsilon}'(h)\bigr)^{2}=\frac{k}{4(1+\varepsilon-h)}
\]
we have
\begin{align}
\sum^{3}_{i=1}|\eta_{i}(v)|^{2} & =2\bigl(a_{\varepsilon}'(h)\bigr)^{2}\bigl(dh(v)\bigr)^{2}\label{eq:noise-assisted-eta-sum}\\
 & \le\frac{k(1-h^{2})|v|^{2}}{2(1+\varepsilon-h)}
\end{align}

Moreover, ${\rm Hess}h=-hg$, and hence, using \eqref{eq:noise-assisted-bar-f},
\begin{align*}
g(\nabla_{v}\overline{f},v) & =g(\frac{3k}{4}\nabla_{v}\nabla h,v)\\
 & =-\frac{3k}{4}{\rm Hess}h(v,v)\\
 & =-\frac{3k}{4}h|v|^{2}.
\end{align*}
Finally, the unit sphere $S^{2}$ has Ricci curvature ${\rm Ric}=g$.
Therefore 
\begin{align}
\sum^{3}_{i=1}R(\sigma_{i},v,v,\sigma_{i}) & =a^{2}_{\varepsilon}|v|^{2}=k(1+\varepsilon-h)|v|^{2}\label{eq:noise-assisted-curvature}
\end{align}
Substitution into \eqref{eq:thm-coup-mani} gives 
\begin{align*}
 & g(\nabla_{v}\overline{f},v)+\frac{1}{2}\sum^{3}_{i=1}\left\{ -R(\sigma_{i},v,v,\sigma_{i})+|\eta_{i}(v)|^{2}\right\} \\
 & \quad\le\left[-\frac{3k}{4}h-\frac{k}{2}(1+\varepsilon-h)+\frac{k}{4}\frac{1-h^{2}}{1+\varepsilon-h}\right]|v|^{2}.
\end{align*}
Since $h\in[-1,1]$, 
\[
\frac{1-h^{2}}{1+\varepsilon-h}=\frac{(1-h)(1+h)}{1+\varepsilon-h}\le1+h.
\]
Consequently, 
\begin{align*}
 & g(\nabla_{v}\overline{f},v)+\frac{1}{2}\sum^{3}_{i=1}\left\{ -R(\sigma_{i},v,v,\sigma_{i})+|\eta_{i}(v)|^{2}\right\} \\
 & \quad\le-\left(\frac{k}{4}+\frac{k\varepsilon}{2}\right)|v|^{2}.
\end{align*}
The contraction estimate \eqref{eq:noise-assisted-contraction} therefore
follows from Theorem~\ref{thm:coup-mani}.

(ii) The infinitesimal generator of \eqref{eq:noise-assisted-sphere}
is 
\begin{equation}
L\varphi=\frac{3k}{4}\left\langle \nabla h,\nabla\varphi\right\rangle +\frac{1}{2}q\Delta\varphi,\label{eq:noise-assisted-generator}
\end{equation}
where $q:=k(1+\varepsilon-h)$ whose adjoint operator is
\[
L^{*}\rho=-{\rm div}\left(\frac{3k}{4}\rho\nabla h\right)+\frac{1}{2}\Delta(\rho q).
\]
To show $\pi_{\epsilon}=\rho_{\epsilon}d{\rm vol}(x)$ is an invariant
measure, it suffices to show that $L^{*}\rho_{\epsilon}=0$. Note
that
\begin{align*}
\nabla(\rho_{\epsilon}q) & =\nabla\left(\frac{k}{Z_{\varepsilon}}(1+\varepsilon-h)^{-3/2}\right)=\frac{3k}{2}\rho_{\epsilon}\nabla h\\
 & \implies\frac{1}{2}\Delta(\rho_{\epsilon}q)={\rm div}\left(\frac{3k}{4}\rho_{\epsilon}\nabla h\right)\\
 & \implies L^{*}\rho_{\epsilon}=0,
\end{align*}
as desired.

(iii) Use the azimuthal coordinate $(\theta,\varphi)$ on the sphere,
with $d_{g}(x,x_{\star})=\theta\in[0,\pi]$ and volume form $d{\rm vol}=\sin\theta d\theta d\phi$.

We shall estimate
\[
W^{2}_{2}(\pi_{\epsilon},\delta_{x_{*}})=\mathbb{E}_{\pi_{\epsilon}}[d^{2}_{g}(x,x_{\star})]=\mathbb{E}_{\pi_{\epsilon}}\theta^{2}.
\]
Note $h(x)=x^{\top}_{\star}x=\cos\theta$. For $\theta\in[0,\pi]$,
we have
\[
\theta^{2}\le\frac{\pi^{2}}{2}(1-\cos\theta).
\]
Let $p=1-\cos\theta$. It suffices to estimate $\mathbb{E}_{\pi_{\epsilon}}[p]$.
By \eqref{eq:noise-assisted-invariant},
\[
\pi_{\epsilon}=\rho_{\epsilon}d{\rm vol}(x)=\frac{1}{Z_{\epsilon}}(\epsilon+p)^{-5/2}\sin\theta d\theta d\phi.
\]
Thus
\begin{align*}
\mathbb{E}_{\pi_{\epsilon}}p & =2\pi\int^{\pi}_{0}p\frac{1}{Z_{\epsilon}}(\epsilon+p)^{-5/2}\sin\theta d\theta\\
 & =\frac{2\pi}{Z_{\epsilon}}\int^{2}_{0}p(\epsilon+p)^{-5/2}dp\\
 & =\frac{\int^{2}_{0}p(\epsilon+p)^{-5/2}dp}{\int^{2}_{0}(\epsilon+p)^{-5/2}dp}\\
 & =2\epsilon+O(\epsilon^{3/2})
\end{align*}

It follows that 
\[
W_{2}(\pi_{\varepsilon},\delta_{x_{\star}})=O(\sqrt{\varepsilon}),
\]
which proves \eqref{eq:pi-concentration}.

\bibliographystyle{siam}
\bibliography{OT,GeometricControl}

\begin{thebibliography}{10}

\bibitem{ahmad2008stochastic}
{\sc S.~S. Ahmad, N.~Rajan, and S.~Raha}, {\em The stochastic logarithmic norm for stability analysis of stochastic differential equations}, arXiv preprint arXiv:0809.0062,  (2008).

\bibitem{aminzare2022stochastic}
{\sc Z.~Aminzare}, {\em Stochastic logarithmic lipschitz constants: A tool to analyze contractivity of stochastic differential equations}, IEEE control systems letters, 6 (2022), pp.~2311--2316.

\bibitem{arnold1974stochastic}
{\sc L.~Arnold}, {\em Stochastic differential equations}, New York : Wiley, 1974.

\bibitem{bakry1985diffusions}
{\sc D.~Bakry and M.~{\'E}mery}, {\em Diffusions hypercontractives}, in S{\'e}minaire de Probabilit{\'e}s XIX, 1983/84, vol.~1123 of Lecture Notes in Mathematics, Springer, 1985, pp.~177--206.

\bibitem{baudoin2016wasserstein}
{\sc F.~Baudoin}, {\em Wasserstein contraction properties for hypoelliptic diffusions}, arXiv preprint arXiv:1602.04177,  (2016).

\bibitem{bouvrie2019wasserstein}
{\sc J.~Bouvrie and J.-J. Slotine}, {\em Wasserstein contraction of stochastic nonlinear systems}, arXiv preprint arXiv:1902.08567,  (2019).

\bibitem{bulloBook2026}
{\sc F.~Bullo}, {\em Contraction Theory for Dynamical Systems}, Kindle Direct Publishing, {1.3}~ed., 2026.

\bibitem{cheng2021exponential}
{\sc L.-J. Cheng, A.~Thalmaier, and S.-Q. Zhang}, {\em Exponential contraction in wasserstein distance on static and evolving manifolds}, Revue Roumaine de Math{\'e}matiques Pures et Appliqu{\'e}es, 66 (2021), pp.~107--129.

\bibitem{crampin1986applicable}
{\sc M.~Crampin and F.~Pirani}, {\em Applicable Differential Geometry}, vol.~59, Cambridge University Press, 1986.

\bibitem{dani2014observer}
{\sc A.~P. Dani, S.-J. Chung, and S.~Hutchinson}, {\em Observer design for stochastic nonlinear systems via contraction-based incremental stability}, IEEE Transactions on Automatic Control, 60 (2014), pp.~700--714.

\bibitem{davydov2022non}
{\sc A.~Davydov, S.~Jafarpour, and F.~Bullo}, {\em Non-euclidean contraction theory for robust nonlinear stability}, IEEE Transactions on Automatic Control, 67 (2022), pp.~6667--6681.

\bibitem{demidovich1967lectures}
{\sc B.~P. Demidovich}, {\em Lectures on stability theory}, 1967.

\bibitem{forni2013differential}
{\sc F.~Forni and R.~Sepulchre}, {\em A differential {L}yapunov framework for contraction analysis}, IEEE Transactions on Automatic Control, 59 (2014), pp.~614--628.

\bibitem{ikeda1989stochastic}
{\sc N.~Ikeda and S.~Watanabe}, {\em Stochastic differential equations and diffusion processes}, vol.~24, North-Holland Publishing Company, 1989.

\bibitem{kawano2019contraction}
{\sc Y.~Kawano, B.~Besselink, and M.~Cao}, {\em Contraction analysis of monotone systems via separable functions}, IEEE Transactions on Automatic Control,  (2019).

\bibitem{kawano2026incremental}
{\sc Y.~Kawano, S.~Betteti, A.~Davydov, and F.~Bullo}, {\em Incremental input-to-state stability and equilibrium tracking for stochastic contracting dynamics}, arXiv preprint arXiv:2602.18382,  (2026).

\bibitem{kawano2023contraction}
{\sc Y.~Kawano and Y.~Hosoe}, {\em Contraction analysis of discrete-time stochastic systems}, IEEE Transactions on Automatic Control, 69 (2023), pp.~982--997.

\bibitem{kawano2025contraction}
\leavevmode\vrule height 2pt depth -1.6pt width 23pt, {\em Contraction analysis of differentially positive discrete-time stochastic systems}, IEEE Transactions on Automatic Control, 70 (2025), pp.~5238--5253.

\bibitem{Koo2026}
{\sc S.~Koo, J.~G. Lee, and H.~Shim}, {\em Continuous-time stochastic contraction induced by common white noise}, Preprint,  (2026).

\bibitem{kunita1990stochastic}
{\sc H.~Kunita}, {\em Stochastic flows and stochastic differential equations}, vol.~24, Cambridge university press, 1990.

\bibitem{Kushner1967}
{\sc H.~J. Kusher}, {\em Stochastic Stability and Control}, Academic Press, 1967.

\bibitem{kuwada2010duality}
{\sc K.~Kuwada}, {\em Duality on gradient estimates and wasserstein controls}, Journal of Functional Analysis, 258 (2010), pp.~3758--3774.

\bibitem{lang1993real}
{\sc S.~Lang}, {\em Real and Functional Analysis}, vol.~142 of Graduate Texts in Mathematics, Springer, 3~ed., 1993.

\bibitem{li1992stochastic}
{\sc X.-M. Li}, {\em Stochastic Flows on Non-compact Manifolds}, PhD thesis, University of Warwick, 1992.
\newblock Available as arXiv:2105.15017.

\bibitem{lohmiller1998contraction}
{\sc W.~Lohmiller and J.-J.~E. Slotine}, {\em On contraction analysis for non-linear systems}, Automatica, 34 (1998), pp.~683--696.

\bibitem{lohmiller1999}
{\sc W.~S. Lohmiller}, {\em Contraction analysis of nonlinear systems}, PhD thesis, Massachusetts Institute of Technology, 1999.

\bibitem{manchester2017control}
{\sc I.~R. Manchester and J.-J.~E. Slotine}, {\em Control contraction metrics: Convex and intrinsic criteria for nonlinear feedback design}, IEEE Transactions on Automatic Control, 62 (2017), pp.~3046--3053.

\bibitem{MONMARCHE2023}
{\sc P.~Monmarché}, {\em Almost sure contraction for diffusions on rd. application to generalized langevin diffusions}, Stochastic Processes and their Applications, 161 (2023), pp.~316--349.

\bibitem{pascu2016shy}
{\sc M.~N. Pascu and I.~Popescu}, {\em Shy and fixed-distance couplings of {B}rownian motions on manifolds}, Stochastic Processes and their Applications, 126 (2016), pp.~628--650.

\bibitem{pham2013stochastic}
{\sc Q.-C. Pham and J.-J. Slotine}, {\em Stochastic contraction in riemannian metrics}, arXiv preprint arXiv:1304.0340,  (2013).

\bibitem{pham2009contraction}
{\sc Q.-C. Pham, N.~Tabareau, and J.-J. Slotine}, {\em A contraction theory approach to stochastic incremental stability}, IEEE Transactions on Automatic Control, 54 (2009), pp.~816--820.

\bibitem{Pro1992}
{\sc P.~Protter}, {\em Stochastic Integration and Differential Equation}, Springer-Verlag, Berlin, Heidelberg, second~ed., 1992.

\bibitem{reyes2017tracking}
{\sc R.~Reyes-B{\'a}ez, A.~van~der Schaft, and B.~Jayawardhana}, {\em Tracking control of fully-actuated port-hamiltonian mechanical systems via sliding manifolds and contraction analysis}, IFAC-PapersOnLine, 50 (2017), pp.~8256--8261.

\bibitem{reyes2018virtual}
\leavevmode\vrule height 2pt depth -1.6pt width 23pt, {\em Virtual differential passivity based control for tracking of flexible-joints robots}, IFAC-PapersOnLine, 51 (2018), pp.~169--174.

\bibitem{ruffer2013convergent}
{\sc B.~S. R{\"u}ffer, N.~Van De~Wouw, and M.~Mueller}, {\em Convergent systems vs. incremental stability}, Systems \& Control Letters, 62 (2013), pp.~277--285.

\bibitem{Russo2010}
{\sc G.~Russo, M.~di~Bernardo, and E.~D. Sontag}, {\em Global entrainment of transcriptional systems to periodic inputs}, PLOS Computational Biology, 6 (2010), pp.~1--26.

\bibitem{sanfelice2011convergence}
{\sc R.~G. Sanfelice and L.~Praly}, {\em Convergence of nonlinear observers on $\mathbb{R}^n$ with a riemannian metric (part i)}, IEEE Transactions on Automatic Control, 57 (2011), pp.~1709--1722.

\bibitem{simpson2014contraction}
{\sc J.~W. Simpson-Porco and F.~Bullo}, {\em Contraction theory on {R}iemannian manifolds}, Systems \& Control Letters, 65 (2014), pp.~74--80.

\bibitem{sontag2010contractive}
{\sc E.~D. Sontag}, {\em Contractive systems with inputs}, in Perspectives in mathematical system theory, control, and signal processing, Springer, 2010, pp.~217--228.

\bibitem{thanwerdas2023}
{\sc Y.~Thanwerdas and X.~Pennec}, {\em {$O(n)$-Invariant Riemannian Metrics on SPD Matrices}}, Linear Algebra and its Applications, 661 (2023), pp.~163--201.

\bibitem{villani2021topics}
{\sc C.~Villani}, {\em Topics in Optimal Transportation}, vol.~58, American Mathematical Soc., 2021.

\bibitem{vonrenesse2005transport}
{\sc M.-K. von Renesse and K.-T. Sturm}, {\em Transport inequalities, gradient estimates, entropy and ricci curvature}, Communications on Pure and Applied Mathematics, 58 (2005), pp.~923--940.

\bibitem{wang2020exponential}
{\sc F.-Y. Wang}, {\em Exponential contraction in wasserstein distances for diffusion semigroups with negative curvature}, Potential Analysis, 53 (2020), pp.~1123--1144.

\bibitem{Wu2021}
{\sc D.~Wu and G.-R. Duan}, {\em Further geometric and {L}yapunov characterizations of incrementally stable systems on {F}insler manifolds}, IEEE Transactions on Automatic Control, 67 (2021), pp.~5614--5621.

\bibitem{wu2024stability}
{\sc D.~Wu, B.~Yi, and A.~Rantzer}, {\em Stability analysis of trajectories on manifolds with applications to observer and controller design}, IEEE Transactions on Automatic Control, 69 (2024), pp.~7104--7111.

\end{thebibliography}

\end{document}